\documentclass[12pt]{article}
\usepackage{amsmath,amsthm,amssymb}
\usepackage{times}
\usepackage{enumerate}
\usepackage{amsfonts}
\usepackage{fancyhdr}
\usepackage{titlesec}
\usepackage{cite}
\usepackage{ifthen}%\usepackage{hyperref}
\usepackage{amssymb}
\usepackage{fancyhdr}
\usepackage{titlesec}
\usepackage[arrow,matrix]{xy}
\usepackage{pifont}
\usepackage{stmaryrd}
\usepackage{setspace}
\usepackage{indentfirst}
\usepackage{amsmath,amssymb,amscd,bbm,amsthm,mathrsfs,dsfont}
\theoremstyle{definition}
\newtheorem{theorem}{Theorem}[section]
\newtheorem{corollary}[theorem]{Corollary}
\newtheorem{lemma}[theorem]{Lemma}

\newtheorem{notation}[theorem]{Notation}
\numberwithin{equation}{section}

\newcommand{\subjclass}[1]{\bigskip\noindent\emph{2010 Mathematics Subject Classification:}\enspace#1}
\newcommand{\keywords}[1]{\noindent\emph{Keywords:}\enspace#1}
\begin{document}

%%%%% To ease editing, add:

\baselineskip=16pt

%%%%%%%%%%%%%%%%

\title{Global well-posedness and   decay estimates  for three-dimensional compressible Navier-Stokes-Allen-Cahn system}

\author{Xiaopeng Zhao \\
\small{ College of Sciences, Northeastern University, Shenyang 110004, China
 }\\
\small{zhaoxiaopeng@jiangnan.edu.cn}
 }

\date{}

\maketitle

%%%%%%%%we study a fourth order parabolic equation with nonlinear principal part modeling epitaxial thin film growth

\begin{abstract}
We study the small data global well-posedness and time-decay rates of solutions to the Cauchy problem for 3D  compressible Navier-Stokes-Allen-Cahn equations  via a refined pure energy method. In particular, the optimal decay rates of the higher-order spatial derivatives of the solution are obtained, the $\dot{H}^{-s}$ ($0\leq s<\frac32$)   negative Sobolev norms is shown to be preserved along time evolution and enhance the decay rates.


\end{abstract}

{\small
\section{Introduction} \label{sect1}
A fluid-mechanical theory for two-phase mixtures of fluids faces a well-known
mathematical difficulty:  the movement of the interfaces is naturally amenable to a
Lagrangian description, while the bulk fluid flow is usually considered in the Eulerian
framework \cite{EF}.  Recently,  the phase-field methods, or sometimes called the diffuse interface approaches, was introduced to overcome this difficulty by postulating the existence of a ``diffuse" interface spread over a possibly narrow region covering the
``real" sharp interface boundary. Now, this models become one of the major tools to study a variety of interfacial phenomena.
As the underlying physical problem still conceptually consists of sharp interfaces, the dynamics of the phase variable remains to a considerable extent purely fictitious \cite{EF}. Typically, different variants of Allen-Cahn, Cahn-Hilliard or other types of dynamics were used to describe the models, see previous studies \cite{ADM,MEG,JL,CL} for example.

One of the
well-known diffuse interface model for two-phase flow  is the coupled Navier-Stokes/Allen-Cahn system. In this model, the interfaces between the phases are assumed to be of ``diffuse" nature, that is, sharp interfaces are replaced by narrow transition layers. These regions are located by a phase field variable $\chi$ governed by the Allen-Cahn equation, while the dynamics are described by the Navier-Stokes equations. In \cite{BT}, Blesgen proposed this model and  developted a thermodynamically and mechanically consistent set of PDEs extending the Navier-Stokes equations to a compressible binary Allen-Cahn mixture. As pointed out in Blesgen \cite{BT}, the  Navier-Stokes-Allen-Cahn model
can be seen as a first step towards incorporating transport mechanism into the
description of phase-formation processes.

During the past years, many authors studied the properties of solutions for the incompressible Navier-Stokes system with matched density. For example, Favre and Schimperna  \cite{FG} considered the existence of weak solution in 3D and well-posedness of strong solution in 2D,
Xu et. al. \cite{Xu} studied the existence of axisymmetric solutions, Zhao et. al. \cite{Zhao} studied the vanishing viscosity limit, Gal and Grasselli \cite{Gal1,Gal2} invistigated the asymptotic behavior and attractors and the reference therein. For the system based on the incompressible Navier-Stokes system with  different
densities, Li et. al. \cite{Li1,Li2} studied the existence and uniqueness of  local strong solutions   and established a blow-up criterion for
such strong solutions in 3D case. Moreover, by using an energetic variational approach, Jiang et. al. \cite{Jiang} derived a different model of  Navier-Stokes-Allen-Cahn, then proved
the existence of weak solutions in 3D, the well-posedness of strong solutions in 2D, and studied the long time behavior
of the strong solutions.

As far as we know, there are also some classical results are available for the initial-boundary value problem of compressible Navier-Stokes-Allen-Cahn system. In \cite{EF}, for the initial-boundary value problem of 3D compressible model, by using Faedo-Galerkin approximation, Feireisl et. al. proved the existence of global-in-time weak solutions without any restriction on the size of initial data for the exponent of pressure $\gamma>6$. This result was extended to $\gamma>2$ in Chen et. al. \cite{SChen}. Moreover, supposed that the boundary conditions of the model  are of mixed type (Neumann-Dirichlet) and may be nonhomogeneous, the density is H\"{o}lder continuous for instance, Kotschote \cite{MK}   established the result on  local-in-time existence and uniqueness statement for sufficiently smooth data, in a general $C^2$ bounded domain of $\mathbb{R}^n$ ($n\geq1$). Ding et. al. \cite{DSJ1,DSJ2} proved the existence and uniqueness of global classical solution, the existence of weak solutions and the existence of unique global strong solution of the initial-boundary value problem of one-dimensional compressible Navier-Stokes-Allen-Cahn system   for the initial data   without vacuum states. Recently, Zhu and his cooperators \cite{HYin,TL}   paid their attentions on the large time behavior of solutions for the  Cauchy problem and inflow problem of 1D compressible Navier-Stokes-Allen-Cahn equations, respectively. In \cite{TL}, Luo, Yin and Zhu    proved that the solutions to the Cauchy problem of 1D compressible Navier-Stokes/Allen-Cahn system tend time-asymptotically to the rarefaction wave, where the strength of the rarefaction wave is not required to be small. In addition,  for the inflow problem of 1D compressible Navier-Stokes/Allen-Cahn system, Yin and Zhu \cite{HYin} analyzed the large-time behavior of the solution, proved the existence of the stationary solution and the asymptotic stability of the nonlinear wave.
However,
to our knowledge, there's no result on the  Cauchy  problem for the 3D compressible Navier-Stokes-Allen-Cahn system.

In this paper, we consider the Cauchy problem of 3D compressible Navier-Stokes-Allen-Cahn equations \cite{BT,MK,GW}
\begin{equation} \label{1-1}
 \left\{ \begin{aligned}
 &\rho_t+\nabla\cdot(\rho u)=0,\\
         &\partial_t(\rho u)+\nabla\cdot(\rho u\otimes u)+\nabla p=\mu\Delta u+(\mu+\lambda)\nabla\hbox{div}u-\ell\nabla\cdot\left(\nabla\chi\otimes\nabla\chi-\frac{|\nabla\chi|^2}2 \mathbb{I}_3\right),\\
         &\partial_t(\rho\chi)+\hbox{div}(\rho\chi u)=-\omega, \\
                   &\rho\omega=-\ell\Delta\chi+\rho\frac{\partial\Phi(\rho,\chi)}{\partial\chi},\\
                   &(\rho,u,\chi)|_{t=0}=(\rho_0,u_0,\chi_0),
                           \end{aligned} \right.
                           \end{equation}
where $\rho$ denotes the total fluid density, $u$ implies the mean velocity of the fluid mixture, $\chi$ is the concentration of one selected constituent and the pressure $p=p(\rho)$ is a smooth function in a neighborhood of $1$ with $p'(1)=1$, respectively. Moreover, $\omega$ is the chemical potential and $\mathbb{I}_3$ denotes a $3\times 3$ identity matrix. $\mu$ and $\lambda$ are two viscosity coefficients, which satisfy
$$
\mu>0,\quad 2\mu+3\lambda\geq0.
$$
The specific free energy $f(\rho,\chi)$ can be defined as
\begin{equation}
\label{1-1x}
\Phi(\rho,\chi)=\frac{\rho^{\gamma-1}}{\gamma-1}+\frac1{\ell}\left(\frac{\chi^4}4-\frac{\chi^2}2\right),
\end{equation}
where $\gamma>1$ is the adiabatic constant and the constant $\sqrt{\ell}$ denotes the thickness of the interfacial region.
In this paper, we take
$$ \rho\frac{\partial \Phi(\rho,\chi)}{\partial\chi}=\frac{\rho}{\ell}(\chi^3-\chi).
$$
For simplicity, we let $\ell\equiv1$ throughout the rest of this paper.
\begin{notation}
In the following, we  use $H^k(\mathbb{R}^3$ $(k\in\mathbb{R})$ to denote the usual Sobolev spaces with norm $\|\cdot\|_{H^s}$ and $L^p(\mathbb{R}^3)$, $1\leq p\leq\infty$ to denote the ususl $L^p$ spaces with norm $\|\cdot\|_{L^p}$. We also introduce the homogeneous negative index Sobolev space $\dot{H}^{-s}(\mathbb{R}^3)$:
$$
\dot{H}^{-s}(\mathbb{R}^3):=\{f\in L^2(\mathbb{R}^3):\||\xi|^{-s}\hat{f}(\xi)\|_{L^2}<\infty\}
$$endowed with the norm $\|f\|_{\dot{H}^{-s}}:=\||\xi|^{-s}\hat{f}(\xi)\|_{L^2}$. The symbol $\nabla^l$ with an integer $l\geq0$ stands for the usual spatial derivatives of order $l$. For instance,  define
 $$
 \nabla^lz=\{\partial_x^{\alpha}z_i||\alpha|=l,i=1,2,3\},~~~z=(z_1,z_2,z_3).
 $$
 If $l<0$ or $l$ is not a positive integer, $\nabla^l$ stands for $\Lambda^l$ defined by
\begin{equation}\label{A9}
\Lambda^sf(x)=\int_{\mathbb{R}^3}|\xi|^s\hat{f}(\xi)e^{2\pi ix\cdot\xi}d\xi,
\end{equation}
where $\hat{f}$ is the Fourier transform of $f$. Moreover, we use the notation $A\lesssim B$ to mean that $A\leq cB$ for a universal constant $c>0$ that only depends on the parameters coming from the problem and the indexes $N$ and $s$  coming from the regularity on the data.
\end{notation}

%The main purpose of this paper is to study the small initial data global well-posedness and optimal decay estimates of strong solutions  for system (\ref{1-1}) in the whole space $\mathbb{R}^3$. We use a general energy method, Kato-Ponce inequality together with the Gagliardo-Nirenberg interpolation techniques to obtain the global well-posedness and the optimal time-decay rates of the solution to system (\ref{1-1}) when the initial data is sufficiently small.

For convenience, we denote $\varrho=\rho-1$, rewrite (\ref{1-1}) in the perturbation form as
\begin{equation} \label{1-2}
 \left\{ \begin{aligned}
 &\varrho_t+\hbox{div}u= -\varrho\hbox{div}u-u\cdot\nabla\varrho,\\
         &u_t-\mu\Delta u-(\mu+\lambda)\nabla\hbox{div}u+\nabla\varrho=-u\cdot\nabla u-h(\varrho)( \mu\Delta u+(\mu+\lambda)\nabla\hbox{div}u)
         \\
         &\quad\quad\quad\quad\quad\quad\quad\quad\quad\quad\quad\quad\quad\quad\quad-g(\varrho)\nabla\varrho-\phi(\varrho)\hbox{div} \left(\nabla\chi\otimes\nabla\chi-\frac{|\nabla\chi|^2}2 \mathbb{I}_3\right),\\
         &\chi_t-\Delta\chi =-u\cdot\nabla\chi-\varphi(\varrho)\Delta\chi-\phi(\varrho)(\chi^3-\chi),\\
                   &(\varrho,u,\chi)|_{t=0}=(\varrho_0,u_0,\chi_0)=(\rho_0-1,u_0,\chi_0),
                           \end{aligned} \right.
                           \end{equation}
where
$$
h(\varrho)=\frac{\varrho}{\varrho+1},\quad g(\varrho)=\frac{p'(\varrho+1)}{\varrho+1}-1,\quad\phi(\varrho)=\frac1{\varrho+1},\quad
\varphi(\varrho)=\frac{\varrho(\varrho+2)}{(\varrho+1)^2},%\quad f(\varrho)=2+\frac1{\varrho(\varrho+2)}.
$$

%Consider the
% Cauchy problem of three-dimensional compressible Navier-Stokes-Allen-Cahn system (\ref{1-1}).

For this system, we can show the local existence theorem.
\begin{lemma}
\label{lem1.1}
Let the initial data$(\varrho_0,u_0,  \chi_0)\in H^3(\mathbb{R}^3)\times H^3(\mathbb{R}^3)\times H^{4}(\mathbb{R}^3)$. Then there exit positive constants $\nu_0>0$ and $T>0$ depending only on $\|(\varrho_0,u_0,\chi_0,\nabla\chi_0)\|_{H^3}$ such that the system (\ref{1-2}) has a unique solution
$$
(\varrho,u,\chi,\nabla\chi)\in L^{\infty}(0,T;H^3),\quad(u,\chi,\nabla\chi)\in L^2(0,T;H^4),
$$
satisfying
$$
\|(\varrho,u,\chi,\nabla\chi)(t)\|_{H^3},~~\left(\nu_0\int_0^t\|\nabla(u,\chi,\nabla\chi)(\tau)\|_{H^3}^2d\tau\right)^{\frac12}\leq 2\|(\varrho_0,u_0,\chi_0,\nabla\chi_0)\|_{H^3},\quad\forall t\in[0,T].
$$
\end{lemma}

Lemma \ref{lem1.1} can be proved by using the linearization and iteration technique. We omit the proof here for simplicity. One can refer to \cite{FZ} for a detailed proof of the local existence results for the equations of compressible Hall-magnetohydrodynamic whose techniques can be easily applied to our case for the 3D compressible Navier-Stokes-Allen-Cahn equations.

The main purpose of this paper is to study the small initial data global well-posedness and optimal decay estimates of strong solutions  for system (\ref{1-1}) in the whole space $\mathbb{R}^3$. We use a general energy method, Kato-Ponce inequality and the Gagliardo-Nirenberg interpolation techniques to obtain the global well-posedness and the optimal time-decay rates of the solution to system (\ref{1-1}) when the initial data is sufficiently small.
Next, one state our main results on the global well-posedness and decay estimates of solutions for system (\ref{1-2}).
\begin{theorem}
\label{thm1.1}
Let $N\geq3$, assume that $(\varrho_0,u_0,  \chi_0)\in H^N(\mathbb{R}^3)\times H^N(\mathbb{R}^3)\times H^{N+1}(\mathbb{R}^3)$, and there exists a constant $\delta_0>0$ such that if
\begin{equation}
\label{1-3}
\|\varrho_0\|_{H^3}+\|u_0 \|_{H^3}+\| \chi_0\|_{H^3}+\| \nabla \chi_0\|_{H^3}\leq\delta_0,
\end{equation}
then there exists a unique global solution $(\varrho,u,\chi)$ satisfying that for all $t\geq0$,
\begin{equation}
\label{1-4}\begin{aligned}&
\|\varrho(t)\|_{H^N}^2+\|u(t)\|_{H^N}^2+\| \chi\|_{H^N}^2+\|\nabla \chi\|_{H^N}^2\\
&+\int_0^t(\|\nabla u(s)\|_{H^N}^2+\|   \chi\|_{H^N}^2+\|\nabla \chi\|_{H^N}^2 )ds
\\
&\leq C(\|\varrho_0\|_{H^N}^2+\|u_0\|_{H^N}^2+\| \chi_0\|_{H^N}^2+\| \nabla\chi_0\|_{H^N}^2).\end{aligned}
\end{equation}
If further, $(\varrho, u_0, \chi_0,\nabla\chi_0)\in\dot{H}^{-s}(\mathbb{R}^3)$ for some $s\in[0,\frac32)$, then for all $t\geq0$,
\begin{equation}
\label{1-5}
\|\Lambda^{-s}\varrho(t)\|_{L^2}^2+\|\Lambda^{-s}u(t)\|_{L^2}^2+\|\Lambda^{-s} \chi(t)\|_{L^2}^2+\|\Lambda^{-s} \nabla\chi(t)\|_{L^2}^2\leq C,
\end{equation}
and
\begin{equation}\label{1-6}
\begin{aligned}&
\|\nabla^l\varrho(t)\|_{H^{N-l}}+\|\nabla^lu(t)\|_{H^{N-l}}+\|\nabla^l \chi(t)\|_{H^{N-l}}\\&+ \|\nabla^{l+1} \chi(t)\|_{H^{N-l}}
\leq  C(1+t)^{-\frac{l+1}2},\quad\hbox{for}~l=0,1,\cdots, N-1 .\end{aligned}
\end{equation}\end{theorem}

Note that the Hardy-Littlewood-Sobolev theorem implies that for $p\in(1,2]$,  $L^{p}(\mathbb{R}^3)\subset\dot{H}^{-s}(\mathbb{R}^3)$ with $s=3(\frac1p-\frac12)\in[0,\frac32)$. Then, on the basis of Theorem \ref{thm1.1}, we easily obtain the following corollary of the optimal decay estimates.
\begin{corollary}
\label{cor1.1}Under the assumptions of Theorem \ref{thm1.1}, if we replace the $\dot{H}^{-s}(\mathbb{R}^3)$ assumption by
$$(\varrho,u_0,\chi_0,\nabla\chi_0)\in L^{p}(\mathbb{R}^3),\quad 1< p\leq 2,
$$
 then the following decay estimate holds:
\begin{equation}
\label{1-7}\begin{aligned}&
\|\nabla^l\varrho(t)\|_{H^{N-l}}+\|\nabla^lu(t)\|_{H^{N-l}}+\|\nabla^l \chi(t)\|_{H^{N-l}}\\&+\|\nabla^{l+1} \chi(t)\|_{H^{N-l}} \leq C(1+t)^{-\left[\frac32\left(\frac1p -\frac12\right)+\frac l2\right]},\quad\hbox{for}~l=0,1,\cdots, N-1.\end{aligned}
\end{equation}
\end{corollary}

 It is worth pointing out that the forms of  incompressible liquid crystal system  \cite{Wei,Gao,GBL} and the incompressible Navier-Stokes-Allen-Cahn system are similar. However, the equation of the direction field $d$ in the compressible nematic liquid crystal system is
\begin{equation}\label{com2xx}
\partial_td+u\cdot\nabla d=\Delta d-(d^3-d),
\end{equation}
which is different from   $\chi$ in compressible Navier-Stokes-Allen-Cahn equations:
\begin{equation}\label{com2}
 \partial_t(\rho\chi)+\hbox{div}(\rho\chi u)= \frac1{\rho} \left[ \ell\Delta\chi-\rho(\chi^3-\chi)\right].
\end{equation}
Since the principle part of (\ref{com2}) is a nonlinear function $-\frac{\ell}{\rho}\Delta\chi$, we can't use the tools of the dissipative equation with linear principle part to  study  the properties of this equation. Moreover, the linear part of the derivative of double well potential also bring trouble to us. When we study the problem in bounded domain, it is easy to control the linear term by using Sobolev's embedding theorem and the a priori estimates. However, in the whole space $\mathbb{R}^3$, since the compactness properties of Sobolev's embedding theorem is losing and there's no linear principle part, it is difficult to control this term.
In order to overcome those difficulties, one multiply $\frac1{\rho}$ to both sides of (\ref{com2}), borrow a linear term from the right hand side of (\ref{com2}), rewrite this equation as a second order PDE
\begin{equation}\label{com3}
\chi_t-\Delta\chi =-u\cdot\nabla\chi-\varphi(\varrho)\Delta\chi-\phi(\varrho)(\chi^3-\chi),
\end{equation}
where $\varrho=\rho-1$, $\phi(\varrho)=\frac1{\varrho+1}$ and $
\varphi(\varrho)=\frac{\varrho(\varrho+2)}{(\varrho+1)^2}$. It is worth pointing out that there exists a linear principle part $-\Delta\chi$ in (\ref{com3}). Then, by using pure energy method, we obtain the higher order a priori estimates for strong solutions of system (\ref{1-2}). On the basis of the standard continuity argument, one obtain the global well-posedness. Moreover, one also establish the negative Sobolev norm estimates for the solutions. By this estimates, we obtain the decay rate of higher order derivatives of strong solutions. We remark that since the decay estimate is same to the heat equation, it is optimal.
% On the other hand, the coupling between the Navier-Stokes equations and the Allen-Cahn equation also bring trouble to us, especially the last term of (\ref{1-1})$_1$, which is a strong nonlinear term. In order to overcome this difficulty,  we not only assume   the initial data $\|\varrho\|_{H^3}+\|u_0 \|_{H^3}+\| \chi_0\|_{H^3}$ is sufficiently small, but also assume $\| \nabla \chi_0\|_{H^3}$ is sufficiently small, obtain a priori estimates on $(\varrho, u,\chi,\nabla\chi)$.

The structure of this paper is organized as follows.  In Section 2, we introduce some preliminary results, which are useful to prove our main results. Section 3 is  devoted to establish some   refined energy estimates for the solution. In Section 4, we   derive the evolution of the negative Sobolev norms of the solution. Finally, the proof  of Theorem \ref{thm1.1}  is postponed in Section 5.

\section{Preliminaries}

In this section, we introduce some helpful results in $\mathbb{R}^3$.

First of all,  the following Kato-Ponce inequality   is of great importance in our paper.
\begin{lemma}[\cite{KP}]\label{Kato}
Let $1<p<\infty$, $s>0$. There exists a positive constant $C$ such that
\begin{equation}
\label{K-1}
\|\nabla^s(fg)-f\nabla^sg\|_{L^p}\leq C(\|\nabla f\|_{L^{p_1}}\|\nabla^{s-1}g\|_{L^{p_2}}+\|\nabla^sf\|_{L^{q_1}}\|g\|_{L^{q_4}}),
\end{equation}
and
\begin{equation}
\label{K-2}
\|\nabla^s(fg)\|_{L^p}\leq C(\|f\|_{L^{p_1}}\|\nabla^sg\|_{L^{p_2}}+\|\nabla^sf\|_{L^{q_1}}\|g\|_{L^{q_2}},
\end{equation}
where $p_2,q_2\in(1,\infty)$ satisfying $\frac1p=\frac1{p_1}+\frac1{p_2}=\frac1{q_1}+\frac1{q_2}$.
\end{lemma}

The following Gagliardo-Nirenberg inequality was proved in \cite{Nirenberg}.
\begin{lemma}[\cite{Nirenberg}]
\label{lem2.1A}
Let $0\leq m,\alpha\leq l$, then we have
\begin{equation}\label{2-1A}
\|\nabla^{\alpha}f\|_{L^p}\lesssim\|\nabla^mf\|_{L^q}^{1-\theta}\|\nabla^lf\|_{L^r}^{\theta},
\end{equation}
where $\theta\in[0,1]$ and $\alpha$ satisfies
\begin{equation}\label{2-2A}
\frac{\alpha}3-\frac1p=\left(\frac m3-\frac1q\right)(1-\theta)+\left(\frac l3-\frac1r\right)\theta.
\end{equation}
Here, when $p=\infty$, we require that $0<\theta<1$.
\end{lemma}

We recall the following commutator estimate:
\begin{lemma}[\cite{W}]\label{lemcom}
Let $m\geq1$ be an integer and define the commutator
\begin{equation}\label{co-1}
[\nabla^m,f]g=\nabla^m(fg)-f\nabla^mg.
\end{equation}
Then, the following inequality holds:
\begin{equation}\label{co-2}
\|[\nabla^m,f]g\|_{L^p}\lesssim\|\nabla f\|_{L^{p_1}}\|\nabla^{m-1}g\|_{L^{p_2}}+\|\nabla^mf\|_{L^{p_3}}\|g\|_{L^{p_4}},
\end{equation}
where $p,~p_2,~p_3\in(1,\infty)$ and
\begin{equation}\label{co-3}
\frac1p=\frac1{p_1}+\frac1{p_2}=\frac1{p_3}+\frac1{p_4}.
\end{equation}
\end{lemma}

Wang \cite{W}, Wei, Li and Yao \cite{Y} introduced the following result:
\begin{lemma}
\label{Wlemma}
Suppose that $\|\varrho\|_{L^{\infty}}\leq 1$ and $p>1$. Let $f(\varrho)$ be a smooth function of $\varrho$ with bounded derivatives of any order, then for any integer $n\geq1$, the following inequality holds:
\begin{equation}\label{WY}
\|\nabla^m(f(\varrho))\|_{L^p}\lesssim\|\nabla^m\varrho\|_{L^p}.
\end{equation}\end{lemma}

We also introduce the Hardy-Littlewood-Sobolev theorem, which implies the following $L^p$ type inequality.
\begin{lemma}[\cite{Stein,15}]
\label{lem2.3A}
Let $0\leq s<\frac32$, $1<p\leq 2$ and $\frac12+\frac s3=\frac1p$, then
\begin{equation}
\label{2-4A}
\|f\|_{\dot{H}^{-s}}\lesssim\|f\|_{L^p}.
\end{equation}
\end{lemma}

The  special Sobolev interpolation lemma will be used in the proof of Theorem \ref{thm1.1}.
\begin{lemma}[\cite{W,Stein}]
\label{lem2.2A}
Let $s,k\geq0$ and $l\geq0$, then
\begin{equation}
\label{2-3A}
\|\nabla^lf\|_{L^2}\leq\|\nabla^{l+k}f\|_{L^2}^{ 1-\theta }\|f\|_{\dot{H}^{-s}}^{ \theta },\quad\hbox{with}~\theta=\frac k{l+k+s}.
\end{equation}
\end{lemma}

\section{Energy estimates}

In this section, we  derive the a priori energy estimates for system (\ref{1-2}). Suppose that there exists a small positive constant $\delta>0$ such that
\begin{equation}
\label{2-1}
\sqrt{\mathcal{E}_0^3(t)}=\|\varrho(t)\|_{H^3}+\|u(t)\|_{H^3}+\|\chi(t)\|_{H^3}+\|\nabla\chi(t)\|_{H^3}\leq\delta,
\end{equation}
which, together with Sobolev's inequality, yields directly
$$
\frac12\leq\varrho+1\leq2.
$$
Simple calculations show that
\begin{equation}\label{x1}
|h(\varrho)|,|g(\varrho)| |\leq C|\varrho|,
\end{equation}
and
\begin{equation}\label{x2}
|\phi^{(l)}(\varrho)|,|\varphi^{(l)}(\varrho)|,|h^{(k)}(\varrho)|,|g^{(k)}(\varrho)|\leq C~~\hbox{for~any}~l\geq0,~k\geq1.
\end{equation}

Next, we establish the following energy estimates including $\varrho$, $u$ and $\chi$ themselves.
\begin{lemma}
\label{lem2.1}
If $\sqrt{\mathcal{E}_0^3(t)}\leq\delta$, then for $k=0,1,\cdots, N-1$, we have
\begin{equation}
\label{2-2}
\begin{aligned}
&\frac d{dt}(\|\nabla^k\varrho\|_{L^2}^2+\|\nabla^ku\|_{L^2}^2+\|\nabla^k\chi\|_{L^2}^2+\|\nabla^k\nabla\chi\|_{ L^2}^2)
\\&+C(\|\nabla^{k+1}u\|_{L^2}^2+\|\nabla^{k+1}\chi\|_{L^2}^2+\|\nabla^{k+1}\nabla\chi\|_{L^2}^2)
\\
\lesssim& (\delta+\delta^3)%\left[ \sqrt{\mathcal{E}_0^3(t)}+(\mathcal{E}_0^3(t))^2\right]
(
\|\nabla^{k +1}\varrho\|_{L^2}^2+\|\nabla^{k+1}u\|_{L^2}^2+\|\nabla^{k+1}\chi\|_{L^2}^2+\|\nabla^{k+1}\nabla\chi\|_{L^2}^2).
\end{aligned}
\end{equation}
\end{lemma}
\begin{proof}
Taking $k$-th spatial derivatives to (\ref{1-1})$_1$, (\ref{1-1})$_2$ and (\ref{1-1})$_3$, $k+1$th spatial derivatives to (\ref{1-1})$_3$, multiplying the resulting identities by $\nabla^k\varrho$, $\nabla^ku$, $\nabla^k\chi$ and $\nabla^{k+1}\chi$ respectively, summing up and then integrating over $\mathbb{R}^3$, we derive that
\begin{equation}
\label{2-3}
\begin{aligned}&
\frac12\frac d{dt}\int_{\mathbb{R}^3}(|\nabla^k\varrho|^2+|\nabla^ku|^2+|\nabla^k\chi|^2+|\nabla^{k+1}\chi|^2)dx
\\
&+\int_{\mathbb{R}^3}(|\nabla^{k+1}u|^2+|\nabla^{k+1}\chi|^2+|\nabla^{k+2}\chi|^2)dx
\\
=&\int_{\mathbb{R}^3}\left[\nabla^k(-\varrho\hbox{div}u-u\cdot\nabla\varrho)\cdot\nabla^k\varrho \right.
\\&
\left.-\nabla^k\left[u\cdot\nabla u+h(\varrho)(\mu\Delta u+(\mu+\lambda)\nabla\hbox{div}u)+g(\varrho)\nabla\varrho+\phi(\varrho)\hbox{div}\left(\nabla\chi\otimes\nabla\chi-\frac{|\nabla\chi|^2}2 \mathbb{I}_3\right)\right]\cdot\nabla^ku \right.
\\
&\left.
+\nabla^k\left(-u\cdot\nabla\chi-\varphi(\varrho)\Delta\chi-\phi(\varrho)(\chi^3-\chi) \right)\cdot\nabla^k\chi  \right.
\\
&\left.
+\nabla^{k+1}\left(-u\cdot\nabla\chi-\varphi(\varrho)\Delta\chi-\phi(\varrho)(\chi^3-\chi)\right)\cdot\nabla^{k+1}\chi  \right]dx
\\
=&\sum_{i=1}^{12}I_i.
\end{aligned}
\end{equation}
  The right hand side terms of (\ref{2-3}) will be estimated one by one in the following. The main idea is that we will carefully interpolate the spatial derivatives between the higher-order derivatives and the lower-order derivatives to bound these nonlinear terms by the right hand side of (\ref{2-2}).
First, for $I_1$, by using Kato-Ponce inequality (Lemma \ref{Kato}), Gagliardo-Nirenberg inequality (Lemma \ref{lem2.1A}) and Sobolev embedding theorem, we can estimate as
\begin{equation}
\begin{aligned}\label{2-4cc}
I_1=&-\int_{\mathbb{R}^3}\nabla^k(\varrho\nabla\cdot u)\cdot\nabla^k\varrho dx
\\
\lesssim&\|\nabla^k\varrho\|_{L^6}\|\nabla^k(\varrho\nabla\cdot u)\|_{L^{\frac65}}
\\
\lesssim&\|\nabla^k\varrho\|_{L^6}(\|\nabla^k\varrho\|_{L^3}\|\nabla u\|_{L^2}+\|\varrho\|_{L^3}\|\nabla^{k+1}u\|_{L^2})
\\
\lesssim&\|\nabla^{k+1}\varrho\|_{L^2}\left(\|\nabla^{k+1}\varrho\|_{L^2}^{\frac k{k+\frac12}}\|\Lambda^{\frac12}\varrho\|_{L^2}^{\frac{\frac12}{k+\frac12}}\|\Lambda^{\frac12}u\|_{L^2}^{\frac k{k+\frac12}}\|\nabla^{k+1}u\|_{L^2}^{\frac{\frac12}{k+\frac12}}+\|\Lambda^{\frac12}\varrho\|_{L^2}\|\nabla^{k+1}u\|_{L^2}\right)
\\
\lesssim&(\|\Lambda^{\frac12}\varrho\|_{L^2}+\|\Lambda^{\frac12}u\|_{L^2})(\|\nabla^{k+1}\varrho\|_{L^2}+\|\nabla^{k+1}u\|_{L^2})\\
\lesssim&\delta(\|\nabla^{k+1}\varrho\|_{L^2}+\|\nabla^{k+1}u\|_{L^2}).
\end{aligned}\end{equation}
Similarly, by using Kato-Ponce inequality (Lemma \ref{Kato}), Gagliardo-Nirenberg inequality (Lemma \ref{lem2.1A}) and Sobolev embedding theorem again, we estimate the terms $I_2$ and $I_3$ as
\begin{equation}
\label{2-4-a1}\begin{aligned}
I_2=&- \int_{\mathbb{R}^3}\nabla^k(u\cdot\nabla\varrho)\cdot\nabla^k\varrho dx
\\
\lesssim&\|\nabla^k\varrho\|_{L^6}\|\nabla^k(u\cdot\nabla\varrho)\|_{L^{\frac65}}
\\
\lesssim&\|\nabla^k\varrho\|_{L^6}(\|\nabla^ku\|_{L^3}\|\nabla\varrho\|_{L^2}+\|u\|_{L^3}\|\nabla^{k+1}\varrho\|_{L^2})
\\
\lesssim&\|\nabla^{k+1}\varrho\|_{L^2}\left(\|\nabla^{k+1}u\|_{L^2}^{\frac k{k+\frac12}}\|\Lambda^{\frac12}u\|_{L^2}^{\frac{\frac12}{k+\frac12}}\|\Lambda^{\frac12}\varrho\|_{L^2}^{\frac k{k+\frac12}}\|\nabla^{k+1}\varrho\|_{L^2}^{\frac{\frac12}{k+\frac12}}+\|\Lambda^{\frac12}u\|_{L^2}\|\nabla^{k+1}\varrho\|_{L^2}\right)
\\
\lesssim&(\|\Lambda^{\frac12}\varrho\|_{L^2}+\|\Lambda^{\frac12}u\|_{L^2})(\|\nabla^{k+1}\varrho\|_{L^2}+\|\nabla^{k+1}u\|_{L^2})\\
\lesssim&\delta(\|\nabla^{k+1}\varrho\|_{L^2}+\|\nabla^{k+1}u\|_{L^2}),
\end{aligned}\end{equation}
and
\begin{equation}
\label{2-4-a2}\begin{aligned}
I_3=&-\int_{\mathbb{R}^3}\nabla^k(u\cdot\nabla u)\cdot\nabla^k udx
\\
\lesssim&\|\nabla^k u\|_{L^6}\|\nabla^k(u\cdot\nabla u)\|_{L^{\frac65}}
\\
\lesssim&\|\nabla^k u\|_{L^6}(\|\nabla^ku\|_{L^3}\|\nabla u\|_{L^2}+\|u\|_{L^3}\|\nabla^{k+1}u\|_{L^2})
\\
\lesssim&\|\nabla^{k+1}u\|_{L^2}\left(\|\nabla^{k+1}u\|_{L^2}^{\frac k{k+\frac12}}\|\Lambda^{\frac12}u\|_{L^2}^{\frac{\frac12}{k+\frac12}}\|\Lambda^{\frac12}u\|_{L^2}^{\frac k{k+\frac12}}\|\nabla^{k+1}u\|_{L^2}^{\frac{\frac12}{k+\frac12}}+\|\Lambda^{\frac12}u\|_{L^2}\|\nabla^{k+1}u\|_{L^2}\right)
\\
\lesssim&\|\Lambda^{\frac12}u\|_{L^2} \|\nabla^{k+1}u\|_{L^2}\\
\lesssim&\delta \|\nabla^{k+1}u\|_{L^2}.
\end{aligned}\end{equation}
Next, for the term $I_4$, we do the approximation to simplify the presentations by
\begin{equation}
\label{2-4-a3}
I_4=-\int_{\mathbb{R}^3}\nabla^k[h(\varrho)(\nu\Delta u+\eta\nabla\hbox{div}u)]\cdot\nabla^kudx\approx-\int_{\mathbb{R}^3}\nabla^k[h(\varrho)\nabla^2u]\cdot\nabla^kudx.
\end{equation}
If $k=0$, on the basis of the fact (\ref{x1}), H\"{o}lder's together with Sobolev embedding theorem, we deduce that
\begin{equation}\label{2-4-a4}
\begin{aligned}
I_4\approx &-\int_{\mathbb{R}^3}h(\varrho)\nabla^2u\cdot udx
\\
\lesssim& \|\nabla h(\varrho)\|_{L^2}\|\nabla u\|_{L^3}\|u\|_{L^6}+\|\varrho\|_{L^{\infty}}\|\nabla u\|_{L^2}\|\nabla u\|_{L^2}
\\
\lesssim&(\|\nabla u\|_{L^3}+\|\varrho\|_{L^{\infty}})(\|\nabla\varrho\|_{L^2}^2+\|\nabla u\|_{L^2}^2)
\\
\lesssim&(\|\Lambda^{\frac32}u\|_{L^2}+\|\nabla\varrho\|_{L^2}^{\frac12}\|\Delta\varrho\|_{L^2}^{\frac12})(\|\nabla\varrho\|_{L^2}^2+\|\nabla u\|_{L^2}^2)\\
\lesssim&\delta(\|\nabla\varrho\|_{L^2}^2+\|\nabla u\|_{L^2}^2).
\end{aligned}\end{equation}
If $k=1$, by the fact (\ref{x1}), integrating by parts, one obtain
\begin{equation}\label{2-4-a5}
\begin{aligned}
I_4\approx &-\int_{\mathbb{R}^3}\nabla[h(\varrho)\nabla^2u]\cdot \nabla udx
\\
\approx&\int_{\mathbb{R}^3}h(\varrho)|\nabla^2u|^2dx
\\
\lesssim& \|h(\varrho)\|_{L^{\infty}}\|\nabla^2u\|_{L^2}^2\\
\lesssim&\|\nabla\varrho\|_{L^2}^{\frac12}\|\Delta\varrho\|_{L^2}^{\frac12}\|\nabla^2u\|_{L^2}^2\\
\lesssim&\delta\|\nabla^2u\|_{L^2}^2.
\end{aligned}\end{equation}
If $k\geq2$, by using  H\"{o}lder's inequality, Lemma \ref{Kato}, Lemma \ref{Wlemma} and Sobolev embedding theorem, we deduce that
\begin{equation}
\label{2-4-a6}\begin{aligned}
I_4
\lesssim&\|\nabla^{k+1} u\|_{L^2}\|\nabla^{k-1}[h(\varrho)\nabla^2u]\|_{L^{2}}
\\
\lesssim&\|\nabla^{k+1}u\|_{L^2}(\|\nabla^{k-1}h(\varrho)\|_{L^6}\|\nabla^2u\|_{L^3}+\|h(\varrho)\|_{L^{\infty}}\|\nabla^{k+1}u\|_{L^2})
\\
\lesssim&\|\nabla^{k+1}u\|_{L^2}(\|\nabla^k\varrho\|_{L^2}\|\Lambda^{\frac52}u\|_{L^2}+\|\nabla\varrho\|_{L^2}^{\frac12}\|\Delta\varrho\|_{L^2}^{\frac12}\|\nabla^{k+1}u\|_{L^2})
\\
\lesssim&\|\nabla^{k+1}u\|_{L^2}\left(\|\nabla^{k+1}\varrho\|_{L^2}^{\frac{k-\frac32}{k-\frac12}}\|\Lambda^{\frac32}\varrho\|_{L^2}^{\frac1{k-\frac12}}\|\nabla^{k+1}u\|_{L^2}^{\frac1{k-\frac12}}\|\Lambda^{\frac32}u\|_{L^2}^{\frac{k-\frac32}{k-\frac12}}+
\|\nabla\varrho\|_{L^2}^{\frac12}\|\Delta\varrho\|_{L^2}^{\frac12}\|\nabla^{k+1}u\|_{L^2}\right)
\\
\lesssim&(\|\Lambda^{\frac32}\varrho\|_{L^2}+\|\Lambda^{\frac32}u\|_{L^2}+\|\nabla\varrho\|_{L^2}+\|\Delta\varrho\|_{L^2})(\|\nabla^{k+1}u\|_{L^2}^2+\|\nabla^{k+1}\varrho\|^2_{L^2})\\
\lesssim&\delta(\|\nabla^{k+1}u\|_{L^2}^2+\|\nabla^{k+1}\varrho\|^2_{L^2}).
\end{aligned}\end{equation}
Moreover,  applying Lemma \ref{Wlemma},  H\"{o}lder's inequality, the Kato-Ponce inequality (Lemma \ref{Kato}) together with Sobolev embedding theorem, the term  $I_5$ can be bounded as
\begin{equation}
\label{2-4-a7}\begin{aligned}
I_5=&-\int_{\mathbb{R}^3}\nabla^k(g(\varrho)\nabla\varrho)\cdot\nabla^k u dx
\\
\lesssim&\|\nabla^ku\|_{L^3}\|\nabla^k(g(\varrho)\nabla\varrho)\|_{L^{\frac32}}
\\
\lesssim&\|\nabla^ku\|_{L^3}(\|\nabla^kg(\varrho)\|_{L^6}\|\nabla\varrho\|_{L^2}+\|g(\varrho)\|_{L^6}\|\nabla^{k+1}\varrho\|_{L^2})
\\
\lesssim&\|\nabla^ku\|_{L^3}(\|\nabla^{k+1}\varrho \|_{L^2}\|\nabla\varrho\|_{L^2}+\|\nabla\varrho\|_{L^2}\|\nabla^{k+1}\varrho\|_{L^2})
\\
\lesssim&\left(\|\nabla^{k+1}u\|_{L^2}^{\frac{k}{k+\frac12}}\|\Lambda^{\frac12}u\|_{L^2}^{\frac{\frac12}{k+\frac12}}\right)
\left(\|\nabla^{k+1} \varrho \|_{L^2} \|\Lambda^{\frac12}\varrho\|_{L^2}^{\frac k{k+\frac12}}\|\nabla^{k+1}\varrho\|_{L^2}^{\frac{\frac12}{k+\frac12}} \right)
\\
\lesssim& (\|\Lambda^{\frac12}\varrho\|_{L^2}+\|\Lambda^{\frac12}u\|_{L^2}) (\|\nabla^{k+1}\varrho\|_{L^2}^2+\|\nabla^{k+1}u\|_{L^2}^2  )\\
\lesssim& \delta (\|\nabla^{k+1}\varrho\|_{L^2}^2+\|\nabla^{k+1}u\|_{L^2}^2 ).
\end{aligned}\end{equation}
  For the term $I_6$, since $|\phi(\varrho)|\leq C$, by using Lemma \ref{Wlemma}, Lemma \ref{Kato} and Soboelv embedding theorem, we have
\begin{equation}
\begin{aligned}\label{2-4}
I_6=&\int_{\mathbb{R}^3}\nabla^k\left[\phi(\varrho)\hbox{div}\left(\nabla\chi\otimes\nabla\chi-\frac{|\nabla\chi|^2}2 \mathbb{I}_3\right)\right]\cdot\nabla^kudx
\\
\lesssim&\|\nabla^k u\|_{L^6}\left\|\nabla^k\left[\phi(\varrho)\hbox{div}\left(\nabla\chi\otimes\nabla\chi-\frac{|\nabla\chi|^2}2 \mathbb{I}_3\right)\right]\right\|_{L^{\frac65}}
\\
\lesssim&\|\nabla^{k+1}u\|_{L^2} \|\phi(\varrho)\|_{L^{\infty}}\left\|\nabla^k\hbox{div}\left(\nabla\chi\otimes\nabla\chi-\frac{|\nabla\chi|^2}2 \mathbb{I}_3\right)\right\|_{L^{\frac65}}
\\
&+\|\nabla^{k+1}u\|_{L^2}\|\nabla^k\phi(\varrho)\|_{L^6}
\left\| \hbox{div}\left(\nabla\chi\otimes\nabla\chi-\frac{|\nabla\chi|^2}2 \mathbb{I}_3\right)\right\|_{L^{\frac 32}}
\\
\lesssim&\|\nabla^{k+1}u\|_{L^2} \|\nabla\chi\|_{L^3}\|\nabla^{k+1}\nabla \chi\|_{L^2}
+\|\nabla^{k+1}u\|_{L^2}\|\nabla^k\varrho\|_{L^6}\|\nabla\chi\|_{L^3}\|\nabla^2\chi\|_{L^3}
\\
\lesssim&\delta^2(\|\nabla^{k+1}u\|_{L^2}^2+ \|\nabla^{k+1}\chi\|_{L^2}^2+\|\nabla^{k+1}\varrho\|_{L^2}^2).
\end{aligned}
\end{equation}
Using Kato-Ponce inequality (Lemma \ref{Kato}), Gagliardo-Nirenberg inequality (Lemma \ref{lem2.1A}) and Sobolev embedding theorem, we can estimate $I_7$ as
\begin{equation}
\begin{aligned}\label{2-5}
I_7=&- \int_{\mathbb{R}^3}\nabla^k(u\nabla\cdot \chi)\cdot\nabla^k\chi dx
\\
\lesssim&\|\nabla^k\chi\|_{L^6}\|\nabla^k(u\nabla\cdot \chi)\|_{L^{\frac65}}
\\
\lesssim&\|\nabla^k\chi\|_{L^6}(\|\nabla^ku\|_{L^3}\|\nabla\chi\|_{L^2}+\|u\|_{L^3}\|\nabla^{k+1}\chi\|_{L^2})
\\
\lesssim&\|\nabla^{k+1}\chi\|_{L^2}\left(\|\nabla^{k+1}u\|_{L^2}^{\frac k{k+\frac12}}\|\Lambda^{\frac12}u\|_{L^2}^{\frac{\frac12}{k+\frac12}}\|\Lambda^{\frac12}\chi\|_{L^2}^{\frac k{k+\frac12}}\|\nabla^{k+1}\chi\|_{L^2}^{\frac{\frac12}{k+\frac12}}+\|\Lambda^{\frac12}u\|_{L^2}\|\nabla^{k+1}\chi\|_{L^2}\right)
\\
\lesssim&(\|\Lambda^{\frac12}u\|_{L^2}+\|\Lambda^{\frac12}\chi\|_{L^2})(\|\nabla^{k+1}\chi\|_{L^2}+\|\nabla^{k+1}u\|_{L^2})\\
\lesssim&\delta(\|\nabla^{k+1}\chi\|_{L^2}+\|\nabla^{k+1}u\|_{L^2}).
\end{aligned}\end{equation} 
For $I_8$, if $k=0$, H\"{o}lder's inequality and Sobolev embedding theorem imply
\begin{equation}
\begin{aligned}\label{2-11x}
I_8\approx&-\int_{\mathbb{R}^3}\varphi(\varrho)\cdot\nabla^2\chi\cdot \chi dx
\\
\lesssim& \|\nabla \varphi(\varrho)\|_{L^2}\|\nabla \chi\|_{L^3}\|\chi\|_{L^6}+\|\varrho\|_{L^6}\|\nabla \chi\|_{L^{\infty}}\|\nabla \chi\|_{L^2}
\\
\lesssim&\delta(\|\nabla\varrho\|_{L^2}^2+\|\nabla\chi\|_{L^2}^2).\end{aligned}
\end{equation}
On the other hand, if $k\geq1$, employing the Leibniz formula and H\"{o}lder's inequality, we arrive at
\begin{equation}
\begin{aligned}
I_8
=& -\int_{\mathbb{R}^3}\nabla^{k-1}[\varphi(\varrho)\nabla^2\chi]\cdot\nabla^{k+1}\chi dx
\\
=& -\sum_{l=0}^{k-1}C_{k-1}^l\int_{\mathbb{R}^3}\nabla^{k-1-l} \varphi(\varrho)\cdot\nabla^{l}\nabla^2\chi\cdot\nabla^{k+1}\chi dx
\\
\lesssim& \sum_{l=0}^{[\frac k2]-1}C_{k-1}^l\|\nabla^{l+2}\chi \|_{L^3}\|\nabla^{k-l-1}\varphi(\varrho)\|_{L^6}\|\nabla^{k+1}\chi \|_{L^2}
\\
&
+\sum_{l=[\frac k2]}^{k-2}C_{k-1}^l\|\nabla^{l+2}\chi \|_{L^6}\|\nabla^{k-1-l}\varphi(\varrho)\|_{L^3} \|\nabla^{k+1}\chi \|_{L^2}+\underbrace{\int_{\mathbb{R}^3}|\varphi(\varrho)||\nabla^{k+1}\chi |^2dx}_{l=k-1}.\label{2-12x}
\end{aligned}
\end{equation}
 Gagliardo-Nirenberg's inequality (Lemma \ref{lem2.1A}) implies that
\begin{equation}
\label{2-13x}
\begin{aligned}&
\sum_{l=0}^{[\frac k2]-1}C_{k-1}^l\|\nabla^{l+2}\chi \|_{L^3}\|\nabla^{k-l-1}\varphi(\varrho)\|_{L^6}\|\nabla^{k+1}\chi \|_{L^2}
\\
\lesssim&\sum_{l=0}^{[\frac k2]-1}\|\nabla^{\alpha}\chi \|_{L^2}^{1-\frac{l+1}{k+1}}\|\nabla^{k+1}\chi \|_{L^2}^{\frac{l+1}{k+1}}\|\varrho\|_{L^2}^{\frac{l+1}{k+1}}\|\nabla^{k+1}\varrho\|_{L^2}^{1-\frac{l+1}{k+1}}
\|\nabla^{k+1}u\|_{L^2}\\
\lesssim&\delta(\|\nabla^{k+1}\varrho\|_{L^2}^2 +\|\nabla^{k+1}\chi \|_{L^2}^2 ),
\end{aligned}
\end{equation}
where $\alpha$ satisfies
$$
\frac{l+2}3-\frac13=\left(\frac{\alpha}3-\frac12\right)\left(1-\frac{l+1}{k+1}\right)+\left(\frac{k+1}3-\frac12\right)\frac{l+1}{k+1},
$$
that is
$$
\alpha=\frac{3k+3}{2k-2l}\in(\frac32,3).
$$
Moreover, also by Gagliardo-Nirenberg's inequality (Lemma \ref{lem2.1A}), we derive that
\begin{equation}
\label{2-14x}
\begin{aligned}&
\sum_{l=[\frac k2]}^{k-2}C_{k-1}^l\|\nabla^{l+2}\chi \|_{L^6}\|\nabla^{k-1-l}\varphi(\varrho)\|_{L^3} \|\nabla^{k+1}\chi \|_{L^2}
\\
\lesssim&\sum_{l=[\frac k2]}^{k-2}\| \chi \|_{L^2}^{1-\frac{l+3}{k+1}}\|\nabla^{k+1}\chi \|_{L^2}^{\frac{l+3}{k+1}}\|\nabla^{\alpha}
\varrho\|_{L^2}^{\frac{l+3}{k+1}}\|\nabla^{k+1}\varrho\|_{L^2}^{1-\frac{l+3}{k+1}}\|\nabla^{k+1}\chi \|_{L^2}
\\
\lesssim&\delta(\|\nabla^{k+1}\varrho\|_{L^2}^2 +\|\nabla^{k+1}\chi \|_{L^2}^2 ),
\end{aligned}
\end{equation}
where $\alpha$ satisfies
$$
\frac{k-1-l}3-\frac13=\left(\frac{\alpha}3-\frac12\right)\frac{l+3}{k+1}+\left(\frac{k+1}3-\frac12\right)\left(1-\frac{l+3}{k+1}\right),
$$
that is
$$
\alpha=\frac{3k+3}{2l+6}\in[\frac32,3).
$$
For the last term of the right hand side of (\ref{2-12x}), we have
\begin{equation}
\label{2-15x}
\int_{\mathbb{R}^3}|\varphi(\varrho)||\nabla^{k+1}\chi |^2dx\leq\|\varphi(\varrho)\|_{L^{\infty}}\|\nabla^{k+1}\chi \|_{L^2}^2\lesssim\delta\|\nabla^{k+1}\chi \|_{L^2}^2.
\end{equation}
Combining (\ref{2-11x})-(\ref{2-15x}) together, we easily obtain
\begin{equation}
\label{2-9}
I_8\lesssim\delta(\|\nabla^{k+1}\varrho\|_{L^2}^2+\|\nabla^{k+1}\chi \|_{L^2}^2).
\end{equation}
Note that $I_9$ satisfies
\begin{equation}
\begin{aligned}\label{2-10}
I_9=&-\int_{\mathbb{R}^3}\nabla^k[\phi(\varrho)(\chi^3-\chi)]\cdot\nabla^k\chi dx
\\
\lesssim&\left(\|\nabla^k[\phi(\varrho) \chi^3 ]\|_{L^{\frac65}}+\|\nabla^k[\phi(\varrho) \chi ]\|_{L^{\frac65}}\right)\|\nabla^k\chi\|_{L^6}
\\
\lesssim&(\|\phi(\varrho)\|_{L^3}\|\nabla^k \chi^3 \|_{L^2}+\|\chi^2 \|_{L^6}\|\chi\|_{L^2}\|\nabla^k\phi(\varrho)\|_{L^6} +\|\nabla^k[\phi(\varrho) \chi ]\|_{L^{\frac65}})\|\nabla^{k+1}\chi\|_{L^2}
\\
\lesssim&(I_{91}+I_{92}+I_{93})\|\nabla^{k+1}\chi\|_{L^2},
\end{aligned}
\end{equation}where we have used H\"{o}lder's inequality and Kato-Ponce inequality (Lemma \ref{Kato}) in (\ref{2-10}). Next, we first estimate $I_{91}$--$I_{92}$ term by term:
\begin{equation}\begin{aligned}
\label{2-11}
I_{91} \lesssim&\|\varrho\|_{L^3}\|\chi\|^2_{L^6}\|\nabla^k\chi\|_{L^6}\\
\lesssim&\|\Lambda^{\frac12}\varrho\|_{L^2}\|\nabla\chi\|_{L^2}^2\|\nabla^{k+1}\chi\|_{L^2}\lesssim\delta^3\|\nabla^{k+1}\chi\|_{L^2},
\end{aligned}\end{equation}and
\begin{equation}\begin{aligned}
\label{2-11d}
I_{92} \lesssim&\|\varrho\|_{L^\infty}\|\chi\|_{L^6}\|\chi\|_{L^2}\|\nabla^k\varrho\|_{L^6}\\
\lesssim&\|\nabla\varrho\|_{L^2}^{\frac12}\|\nabla^2\varrho\|_{L^2}^{\frac12}\|\nabla\chi\|_{L^2}\|\chi\|_{L^2}\|\nabla^{k+1}\varrho\|_{L^2}
\\
\lesssim&\delta^3\|\nabla^{k+1}\varrho\|_{L^2}.
\end{aligned}\end{equation}
 Since $\phi(\varrho)=\frac1{\varrho+1}$, we obtain $\zeta(\varrho):=\sqrt{\phi(\varrho)}=\frac1{\sqrt{\varrho+1}}$. It is easy to see that  $ \zeta(\varrho)$ is a smooth function of $\varrho$ with bounded derivatives of any order. Hence, Lemma \ref{Wlemma} holds for $\zeta(\varrho)$. By using Sobolev embedding theorem and Kato-Ponce inequality, we have
 \begin{equation}
\begin{aligned}\label{2-14-r}
I_{93}=&\|\nabla^k[\phi(\varrho) \chi ]\|_{L^{\frac65}}=\|\nabla^k[\zeta(\varrho)\zeta(\varrho)\chi ]\|_{L^{\frac65}}
\\
\lesssim&\|\zeta(\varrho)\|_{L^2}\|\zeta(\varrho)\|_{L^6}\|\nabla^k\chi\|_{L^6}+\|\chi\|_{L^2}\|\zeta(\varrho)\|_{L^6}\|\nabla^k\zeta(\varrho)\|_{L^6}
\\
\lesssim&\|\varrho\|_{L^2}\|\nabla\varrho\|_{L^2}\|\nabla^{k+1}\chi\|_{L^2}+\|\chi\|_{L^2}\|\nabla\varrho\|_{L^2}\|\nabla^{k+1}\varrho\|_{L^2}\\
\lesssim&  \delta^2 (\|\nabla^{k+1}\varrho\|_{L^2}+\|\nabla^{k+1}\chi\|_{L^2}).
 \end{aligned}
 \end{equation}
Combining (\ref{2-11})-(\ref{2-14-r}) together, we derive that
\begin{equation}
\label{2-15}
I_9\lesssim\left(\delta^3+\delta\right)(\|\nabla^{k+1}\varrho\|_{L^2}^2+\|\nabla^{k+1}\chi\|_{L^2}^2+\|\nabla^{k+2} \chi\|_{L^2}^2).
\end{equation}
The term $I_{10}$ satisfies
\begin{equation}
\label{2-16}
\begin{aligned}
I_{10}=&-\int_{\mathbb{R}^3}\nabla^{k+1}[u\cdot\nabla\chi]\cdot\nabla^{k+1}\chi dx
\\
\lesssim&\|\nabla^{k+1}\chi\|_{L^6}\|\nabla^{k+1}[u\cdot\nabla\chi]\|_{L^{\frac65}}
\\
\lesssim&\|\nabla^{k+1}\chi\|_{L^6}(\|\nabla^{k+1}u\|_{L^2}\|\nabla\chi\|_{L^3}+\|u\|_{L^3}\|\nabla^{k+1}\nabla\chi\|_{L^2})
\\
\lesssim&(\|\nabla\chi\|_{L^3}+\|u\|_{L^3})(\|\nabla^{k+1}\nabla\chi\|_{L^2}^2+\|\nabla^{k+1}u\|_{L^2}^2)
\\
\lesssim&\delta(\|\nabla^{k+1}\nabla\chi\|_{L^2}^2+\|\nabla^{k+1}u\|_{L^2}^2),
\end{aligned}
\end{equation}
where we have used  Kato-Ponce inequality (Lemma \ref{Kato}) and Sobolev embedding theorem in (\ref{2-16}).
Similar to (\ref{2-16}),  $I_{11}$ and $I_{12}$  can be bounded as
\begin{equation}
\begin{aligned}\label{2-20}
I_{11}=&-\int_{\mathbb{R}^3}\nabla^{k+1}[\varphi(\varrho)\Delta\chi]\cdot\nabla^{k+1}\chi dx
\\
\lesssim&\|\nabla^{k+2}\chi\|_{L^2}\|\nabla^k[\varphi(\varrho)\Delta\chi]\|_{L^2}
\\
\lesssim&\|\nabla^{k+2}\chi\|_{L^2}(\|\varphi(\varrho)\|_{L^{\infty}}\|\nabla^k\Delta\chi\|_{L^2}+\|\Delta\chi\|_{L^{3
}}\|\nabla^k\varphi(\varrho)\|_{L^6})
\\
\lesssim&\|\nabla^{k+2}\chi\|_{L^2}(\|\nabla\varrho\|_{L^{2}}^{\frac12}\|\nabla^2\varrho\|_{L^{2}}
^{\frac12}\|\nabla^{k+2}\chi\|_{L^2}+\|\nabla^3\chi\|_{L^2}^{\frac12}\|\nabla^2\chi\|_{L^2}^{\frac12}\|\nabla^{k+1}\varrho\|_{L^2})
\\
\lesssim&\delta(\|\nabla^{k+2}\chi\|_{L^2}^2+\|\nabla^{k+1}\varrho\|_{L^2}^2),
\end{aligned}
\end{equation}and
\begin{equation}
\begin{aligned}
\label{2-21}
I_{12}=&-\int_{\mathbb{R}^3}\nabla^{k+1}[\phi(\varrho)(\chi^3-\chi)]\cdot\nabla^{k+1}\chi dx
\\
\lesssim&\|\nabla^{k+2}\chi\|_{L^2}(\|\nabla^k[\phi(\varrho) \chi^3 ]\|_{L^2}+\|\nabla^k[\phi(\varrho) \chi  ]\|_{L^2})
\\
\lesssim&\|\nabla^{k+2}\chi\|_{L^2}(\|\phi(\varrho)\|_{L^{\infty}}\|\nabla^k\chi^3\|_{L^2}+\|\chi^3 \|_{L^{3}}
\|\nabla^k\phi(\varrho)\|_{L^6}\\&+\|\zeta(\varrho)\|_{L^2}\|\zeta(\varrho)\|_{L^6}\|\nabla^k\chi\|_{L^6}+\|\chi\|_{L^2}\|\zeta(\varrho)\|_{L^6}\|\nabla^k\zeta(\varrho)\|_{L^6}
)
.
\end{aligned}
\end{equation}
It then follows from (\ref{2-11})  that
\begin{equation}
\label{2-22}\|\nabla^k \chi^3\|_{L^2}\lesssim \delta^2\|\nabla^{k+1}\chi\|_{L^2}.
\end{equation}
Adding (\ref{2-21}) and (\ref{2-22}) together gives
\begin{equation}
\begin{aligned}
\label{2-23}
I_{12}
\lesssim&\|\nabla^{k+2}\chi\|_{L^2}(\|\phi(\varrho)\|_{L^{\infty}}\delta^2\|\nabla^{k+2}\chi\|_{L^2}
+(\|\chi\|_{L^{9}}^3+\|\chi\|_{L^{3}})\|\nabla^{k+1} \varrho \|_{L^2})
\\
\lesssim&\left(\delta^3+\delta\right)(\|\nabla^{k+1}\chi\|_{L^2}^2+\|\nabla^{k+2}\chi\|_{L^2}^2+\|\nabla^{k+1} \varrho \|_{L^2}^2).
\end{aligned}
\end{equation}
Summing up the estimates for $I_1$ -$I_{12}$,  we deduce (\ref{2-2}), this yields the desired result.
\end{proof}
We also need to derive the second type of energy estimates excluding $\varrho$, $u$ and $\chi$ themselves.
\begin{lemma}
\label{lem2.1-1}
If $\sqrt{\mathcal{E}_0^3(t)}\leq \delta$. Then, for $k=0,1,\cdots,N-1$, the following inequality holds:
\begin{equation}
\label{a2-2}
\begin{aligned}
&\frac d{dt}(\|\nabla^{k+1}\varrho\|_{L^2}^2+\|\nabla^{k+1}u\|_{L^2}^2+\|\nabla^{k+1}\chi\|_{L^2}^2+\|\nabla^{k+1}\nabla\chi\|_{L^2}^2)
\\&+C(\|\nabla^{k+2}u\|_{L^2}^2+\|\nabla^{k+2}\chi\|_{L^2}^2+\|\nabla^{k+2}\nabla\chi\|_{L^2}^2)
\\
\lesssim&(\delta^3+\delta)(\|\nabla^{k+1}\varrho\|_{L^2}^2+\|\nabla^{k+2}u\|_{L^2}^2+\|\nabla^{k+2}\chi\|_{L^2}^2+\|\nabla^{k+2}\nabla\chi\|_{L^2}^2).
\end{aligned}\end{equation}
\end{lemma}
\begin{proof}
Taking $k+1$th spatial derivatives to (\ref{1-1})$_1$, (\ref{1-1})$_2$ and (\ref{1-1})$_3$, $k+2$th spatial derivatives to (\ref{1-1})$_3$, multiplying the resulting identities by $\nabla^{k+1}\varrho$, $\nabla^{k+1}u$, $\nabla^{k+1}\chi$ and $\nabla^{k+2}\chi$ respectively, summing up and then integrating over $\mathbb{R}^3$ by parts, we derive that
\begin{equation}
\label{a2-3}
\begin{aligned}&
\frac12\frac d{dt}\int_{\mathbb{R}^3}(|\nabla^{k+1}\varrho|^2+|\nabla^{k+1}u|^2+|\nabla^{k+1}\chi|^2+|\nabla^{k+2}\chi|^2)dx
\\
&+\int_{\mathbb{R}^3}(|\nabla^{k+2}u|^2+|\nabla^{k+2}\chi|^2+|\nabla^{k+3}\chi|^2)dx
\\
=&\int_{\mathbb{R}^3}\left[\nabla^{k+1}(-\varrho\hbox{div}u-u\cdot\nabla\varrho)\cdot\nabla^{k+1}\varrho \right.
\\&
\left.-\nabla^{k+1}\left[u\cdot\nabla u+h(\varrho)(\mu\Delta u+(\mu+\lambda)\nabla\hbox{div}u)+g(\varrho)\nabla\varrho+\phi(\varrho)\hbox{div}\left(\nabla\chi\otimes\nabla\chi-\frac{|\nabla\chi|^2}2 \mathbb{I}_3\right)\right] \cdot\nabla^{k+1}u\right.
\\
&\left.
+\nabla^{k+1}\left(-u\cdot\nabla\chi-\varphi(\varrho)\Delta\chi-\phi(\varrho)(\chi^3-\chi) \right)\cdot\nabla^{k+1}\chi  \right.
\\
&\left.
+\nabla^{k+2}\left(-u\cdot\nabla\chi-\varphi(\varrho)\Delta\chi-\phi(\varrho)(\chi^3-\chi)\right)\cdot\nabla^{k+2}\chi  \right]dx
\\
=&\sum_{i=1}^{12}K_i.
\end{aligned}
\end{equation}
We will estimate the term $K_1$--$K_{11}$ on the right hand side of (\ref{a2-3}) one by one. First, through H\"{o}lder's inequality and Lemma \ref{Kato}, we arrive at
\begin{equation}\begin{aligned}
\label{2-24}
K_1=&- \int_{\mathbb{R}^3}\nabla^{k+1}(\varrho\hbox{div}u)\cdot\nabla^{k+1}\varrho dx
\\
\lesssim&\|\nabla^{k+1}\varrho\|_{L^2}\|\nabla^{k+1}(\varrho\nabla\cdot u)\|_{L^2}
\\
\lesssim&\|\nabla^{k+1}\varrho\|_{L^2}(\|\nabla^{k+1}\varrho\|_{L^2}\|\nabla u\|_{L^{\infty}}+\|\varrho\|_{L^{\infty}}\|\nabla^{k+2}u\|_{L^2})
\\
\lesssim&(\|\nabla u\|_{L^{\infty}}+\|\varrho\|_{L^{\infty}})(\|\nabla^{k+1}\varrho\|_{L^2}^2+\|\nabla^{k+2}u\|_{L^2}^2)\\
\lesssim&\delta(\|\nabla^{k+1}\varrho\|_{L^2}^2+\|\nabla^{k+2}u\|_{L^2}^2).
\end{aligned}
\end{equation}
Next, for the term $J_2$, we utilize the commutator notation (\ref{co-1}) to rewrite it, then integrate by part and use Sobolev's inequality, obtain the following inequality:
\begin{equation}
\label{2-24-a1}
\begin{aligned}
K_2=&- \int_{\mathbb{R}^3}\nabla^{k+1}(u\cdot\nabla\varrho)\cdot\nabla^{k+1}\varrho dx
\\
=&-\int_{\mathbb{R}^3}(u\cdot\nabla\nabla^{k+1}\varrho+[\nabla^{k+1},u]\cdot\nabla\varrho)\nabla^{k+1}\varrho dx
\\
=&-\int_{\mathbb{R}^3}u\cdot\nabla\frac{|\nabla^{k+1}\varrho|^2}2dx+(\|\nabla u\|_{L^{\infty}}\|\nabla^k\nabla\varrho\|_{L^2}+\|\nabla^{k+1}u\|_{L^2}\|\nabla\varrho\|_{L^{\infty}})\|\nabla^{k+1}\varrho\|_{L^2}
\\
=&\frac12\int_{\mathbb{R}^3}\nabla\cdot u|\nabla^{k+1}\varrho|^2dx+(\|\nabla u\|_{L^{\infty}}\|\nabla^k\nabla\varrho\|_{L^2}+\|\nabla^{k+1}u\|_{L^2}\|\nabla\varrho\|_{L^{\infty}})\|\nabla^{k+1}\varrho\|_{L^2}
\\
\lesssim&\|\nabla u\|_{L^{\infty}}\|\nabla^{k+1}\varrho\|^2_{L^2}+(\|\nabla u\|_{L^{\infty}}+\|\nabla\varrho\|_{L^{\infty}})\|\nabla^{k+1}\varrho\|_{L^2}^2\\
\lesssim&\delta\|\nabla^{k+1}\varrho\|^2_{L^2}.
\end{aligned}
\end{equation}
Integrating by parts, applying H\"{o}lder's inequality, Kato-Ponce inequality (Lemma \ref{Kato}),  Gagliardo-Nirenberg inequality (Lemma \ref{lem2.1A}) and Sobolev embedding theorem, we can estimate $K_3$--$K_8$ as
\begin{equation}
\begin{aligned}\label{2-24-a2}
K_3=&-\int_{\mathbb{R}^3}\nabla^{k+1}(u\cdot\nabla u)\cdot\nabla^{k+1}udx
\\
=&\int_{\mathbb{R}^3}\nabla^k(u\cdot\nabla u)\cdot\nabla^{k+2}udx
\lesssim \|\nabla^{k+2}u\|_{L^2}\|\nabla^{k}(u\cdot\nabla u)\|_{L^2}
\\
\lesssim&\|\nabla^{k+2}u\|_{L^2}(\|\nabla^ku\|_{L^6}\|\nabla u\|_{L^3}+\|u\|_{L^3}\|\nabla^{k+1}u\|_{L^6})
\\
\lesssim&\|\nabla^{k+2}u\|_{L^2}\left(\|\nabla^{k+2}u\|_{L^2}^{\frac{k-\frac12}{k+\frac12}}\|\Lambda^{\frac12}u\|_{L^2}^{\frac1{k+\frac12}}\|\Lambda^{\frac12}u\|_{L^2}^{\frac{k-\frac12}{k+\frac12}}\|\nabla^{k+2}u\|_{L^2}^{\frac1{k+\frac12}}+
\|\Lambda^{\frac12}u\|_{L^2}\|\nabla^{k+2}u\|_{L^2}\right)
\\
\lesssim&\|\Lambda^{\frac12}u\|_{L^2}\|\nabla^{k+2}u\|_{L^2}^2\lesssim\delta\|\nabla^{k+2}u\|_{L^2}^2,\end{aligned}
\end{equation}
\begin{equation}\label{2-24-a3}
\begin{aligned}K_4\approx&-\int_{\mathbb{R}^3}\nabla^{k+1}[h(\varrho)\nabla^2u]\cdot\nabla^{k+1}udx
\\
\approx&\int_{\mathbb{R}^3}\nabla^{k}[h(\varrho)\nabla^2u]\cdot\nabla^{k+2}udx
\\
\lesssim&\|\nabla^{k+2}u\|_{L^2}\|\nabla^k[h(\varrho)\nabla^2u]\|_{L^2}
\\
\lesssim&\|\nabla^{k+2}u\|_{L^2}(\|\nabla^kh(\varrho)\|_{L^6}\|\nabla^2u\|_{L^3}+\|h(\varrho)\|_{L^{\infty}}\|\nabla^{k+2}u\|_{L^2})
\\
\lesssim&\|\nabla^{k+2}u\|_{L^2}(\|\nabla^{k+1}\varrho \|_{L^2}\|\nabla^{\frac52}u\|_{L^2}+\|\nabla\varrho\|_{L^2}^{\frac12}\|\Delta\varrho\|_{L^2}^{\frac12} \|\nabla^{k+2}u\|_{L^2})
\\
\lesssim&(\|\nabla^{\frac52}u\|_{L^2}+\|\nabla\varrho\|_{L^2}+\|\Delta\varrho\|_{L^2})(\|\nabla^{k+2}u\|_{L^2}^2+\|\nabla^{k+1}\varrho\|_{L^2}^2)\\
\lesssim&\delta(\|\nabla^{k+2}u\|_{L^2}^2+\|\nabla^{k+1}\varrho\|_{L^2}^2),
\end{aligned}
\end{equation}
\begin{equation}\label{2-24-4}
\begin{aligned}
K_5=&-\int_{\mathbb{R}^3}\nabla^{k+1}[g(\varrho )\nabla\varrho]\cdot\nabla^{k+1}udx
\\
=&\int_{\mathbb{R}^3}\nabla^{k}[g(\varrho)\nabla\varrho]\cdot\nabla^{k+2}udx
\\
\lesssim&\|\nabla^{k+2}u\|_{L^2}(\|\nabla^kg(\varrho)\|_{L^6}\|\nabla\varrho\|_{L^3}+\|g(\varrho)\|_{L^{\infty}}\|\nabla^{k+1}\varrho\|_{L^2})
\\
\lesssim&\|\nabla^{k+2}u\|_{L^2}(\| \nabla^{k+1}\varrho   \|_{L^2}\|\Lambda^{\frac32}\varrho\|_{L^2}+\| \nabla\varrho \|_{L^2}^{\frac12}\| \Delta\varrho \|_{L^2}^{\frac12}\|\nabla^{k+1}\varrho\|_{L^2})
\\
\lesssim&(\|\Lambda^{\frac32}\varrho\|_{L^2}+\| \nabla\varrho \|_{L^2}+\| \Delta\varrho \|_{L^2})(\|\nabla^{k+2}u\|_{L^2}^2+\|\nabla^{k+1}\varrho\|_{L^2}^2 \|_{L^2}^2)\\\lesssim&
\delta(\|\nabla^{k+2}u\|_{L^2}^2+\|\nabla^{k+1}\varrho\|_{L^2}^2 \|_{L^2}^2),
\end{aligned}
\end{equation}
\begin{equation}
\begin{aligned}\label{a2-4}
K_6=&\int_{\mathbb{R}^3}\nabla^{k+1}\left[\phi(\varrho)\hbox{div}\left(\nabla\chi\otimes\nabla\chi-\frac{|\nabla\chi|^2}2 \mathbb{I}_3\right)\right]\cdot\nabla^{k+1}udx
\\
\lesssim&\|\nabla^{k+2} u\|_{L^2}\left\|\nabla^{k }\left[\phi(\varrho)\hbox{div}\left(\nabla\chi\otimes\nabla\chi-\frac{|\nabla\chi|^2}2 \mathbb{I}_3\right)\right]\right\|_{L^{2}}
\\
\lesssim&\|\nabla^{k+2}u\|_{L^2} \|\phi(\varrho)\|_{L^{\infty}}\left\|\nabla^{k }\hbox{div}\left(\nabla\chi\otimes\nabla\chi-\frac{|\nabla\chi|^2}2 \mathbb{I}_3\right)\right\|_{L^{2}}
\\
&+\|\nabla^{k+2}u\|_{L^2}\|\nabla^{k }\phi(\varrho)\|_{L^6}
\left\| \hbox{div}\left(\nabla\chi\otimes\nabla\chi-\frac{|\nabla\chi|^2}2 \mathbb{I}_3\right)\right\|_{L^{3}}
\\
\lesssim&\|\nabla^{k+2}u\|_{L^2}\|\varrho\|_{L^{\infty}}\|\nabla\chi\|_{L^3}\|\nabla^{k+2} \chi\|_{L^2}
+\|\nabla^{k+2}u\|_{L^2}\|\nabla^{k+1}\varrho\|_{L^2}\|\nabla\chi\|_{L^6}\|\nabla^2\chi\|_{L^6}
\\
\lesssim&\delta^2(\|\nabla^{k+2}u\|_{L^2}^2+ \|\nabla^{k+1}\chi\|_{L^2}^2+\|\nabla^{k+2}\varrho\|_{L^2}^2),
\end{aligned}
\end{equation}
\begin{equation}
\begin{aligned}\label{a2-5}
K_7=&-\int_{\mathbb{R}^3}\nabla^{k+1}(u\cdot\nabla\chi)\cdot\nabla^{k+1}\chi dx
\\
\lesssim&\|\nabla^{k+1}\chi\|_{L^6}\|\nabla^{k+1}(u\cdot\nabla\chi)\|_{L^{\frac65}}
\\
\lesssim&\|\nabla^{k+2} \chi\|_{L^2}(\|\nabla^{k+1}u\|_{L^2}\|\nabla\chi\|_{L^3}+\|u\|_{L^3}\|\nabla^{k+1}\nabla\chi\|_{L^2})
\\
\lesssim&\|\nabla^{k+2} \chi\|_{L^2}\left(\|\nabla^{k+2}u\|_{L^2}^{\frac{k+\frac12}{k+\frac32}}\|\Lambda^{\frac12}u\|_{L^2}^{\frac1{k+\frac32}}
\|\nabla^{k+2}\chi\|_{L^2}^{\frac{1}{k+\frac32}}\| \Lambda^{\frac12}\chi\|_{L^2}^{\frac{k+\frac12}{k+\frac32}}+\|\Lambda^{\frac12}u\|_{L^2}\|\nabla^{k+2}\chi \|_{L^2}\right)
\\
\lesssim&(\|\Lambda^{\frac12}\chi\|_{L^2}+\|\Lambda^{\frac12}u\|_{L^2})(\|\nabla^{k+1}\nabla\chi\|_{L^2}^2+\|\nabla^{k+2}u\|_{L^2}^2)
\\
\lesssim&\delta(\|\nabla^{k+2} \chi\|_{L^2}^2+\|\nabla^{k+2}u\|_{L^2}^2),
\end{aligned}
\end{equation} and
\begin{equation}
\label{a2-9}
\begin{aligned}
K_8=&-\int_{\mathbb{R}^3}\nabla^{k+1}[\varphi(\varrho)\Delta\chi]\cdot\nabla^{k+1}\chi dx
\\
\lesssim& \|\nabla^{k+1}[\varphi(\varrho)\Delta\chi]\|_{L^{\frac65}}\|\nabla^{k+1}\chi\|_{L^6}
\\
\lesssim&(\|\varphi(\varrho)\|_{L^3}\|\nabla^{k+1}\Delta\chi\|_{L^2}+\|\Delta\chi\|_{L^3}\|\nabla^{k+1}\varphi(\varrho)\|_{L^2})\|\nabla^{k+2}\chi\|_{L^2}
\\
\lesssim&(\|\varrho\|_{L^3}\|\nabla^{k+3}\chi\|_{L^2}+\|\nabla^2\chi\|_{L^3}\|\nabla^{k+1}\varrho\|_{L^2})\|\nabla^{k+2}\chi\|_{L^2}
\\
\lesssim&\delta(\|\nabla^{k+2}\chi\|_{L^2}^2+\|\nabla^{k+3}\chi\|_{L^2}^2+\|\nabla^{k+1}\varrho\|_{L^2}^2 ).
\end{aligned}\end{equation}
Next, we consider the term $K_9$. H\"{o}lder's inequality   implies that
\begin{equation}
\begin{aligned}\label{a2-10}
K_9=&-\int_{\mathbb{R}^3}\nabla^{k+1}[\phi(\varrho)(\chi^3-\chi)]\cdot\nabla^{k+1}\chi dx
\\
\lesssim&(\|\nabla^{k+1}[\phi(\varrho) \chi^3 ]\|_{L^{\frac65}}+\|\nabla^{k+1}[\zeta(\varrho)\zeta(\varrho) \chi ]\|_{L^{\frac65}})\|\nabla^{k+1}\chi\|_{L^6}
\\
=:&(K_{91}+K_{92})\|\nabla^{k+2}\chi\|_{L^2}.
\end{aligned}\end{equation}
By using Kato-Ponce inequality of Lemma \ref{Kato} and Sobolev embedding theorem, we arrive at
\begin{equation}\begin{aligned}
\label{a2-11}
K_{91}=&\|\nabla^{k+1}[\phi(\varrho) \chi^3 ]\|_{L^{\frac65}}
\\
\lesssim&  \|\phi(\varrho)\|_{L^3}\|\nabla^{k+1} \chi^3 \|_{L^2}+\|\chi^3 \|_{L^3}\|\nabla^{k+1}\phi(\varrho)\|_{L^2}
\\
\lesssim& \|\phi(\varrho)\|_{L^3}\|\chi\|_{L^6}^2\|\nabla^{k+1}\chi\|_{L^6}+\|\chi\|_{L^\infty}\|\chi\|_{L^6}^2\|\nabla^{k+1}\varrho\|_{L^2}
\\
\lesssim& \|\Lambda^{\frac12}\varrho\|_{L^2}\|\nabla\chi\|^2_{L^2}\|\nabla^{k+2}\chi\|_{L^2}+\|\nabla\chi\|_{L^2}^{\frac12}\|\Delta\chi\|_{L^2}^{\frac12}
\|\nabla\chi\|_{L^2}^2\|\nabla^{k+1}\varrho\|_{L^2}
\\
\lesssim&\delta^3(\|\nabla^{k+2}\chi\|_{L^2}+\|\nabla^{k+1}\varrho\|_{L^2} ) .
\end{aligned}
\end{equation}
Moreover, the term $K_{92}$ can be bounded as
 \begin{equation}
\begin{aligned}\label{a2-14-r}
K_{92}=&\|\nabla^{k+1}[\phi(\varrho) \chi ]\|_{L^{\frac65}}=\|\nabla^{k+1}[\zeta(\varrho)\zeta(\varrho)\chi ]\|_{L^{\frac65}}
\\
\lesssim&\|\zeta(\varrho)\|_{L^2}\|\zeta(\varrho)\|_{L^6}\|\nabla^{k+1}
\chi\|_{L^6}+\|\chi\|_{L^6}\|\zeta(\varrho)\|_{L^6}\|\nabla^{k+1}\zeta(\varrho)\|_{L^2}
\\
\lesssim&\|\varrho\|_{L^2}\|\nabla\varrho\|_{L^2}\|\nabla^{k+2}\chi\|_{L^2}+\|\nabla\chi\|_{L^2}\|\nabla\varrho\|_{L^2}\|\nabla^{k+1}\varrho\|_{L^2}\\
\lesssim&  \delta^2 (\|\nabla^{k+1}\varrho\|_{L^2}+\|\nabla^{k+2}\chi\|_{L^2}).
 \end{aligned}
 \end{equation}
Combining (\ref{a2-11})-(\ref{a2-14-r}) together, we derive that
\begin{equation}
\label{a2-15}
K_9\lesssim\left(\delta^3+\delta^2\right)(\|\nabla^{k+1}\varrho\|_{L^2}^2+\|\nabla^{k+2}\chi\|_{L^2}^2 ).
\end{equation}
Next,   employing H\"{o}lder's inequality, Kato-Ponce inequality of Lemma \ref{Kato} and Sobolev embedding theorem, we estimate the term $K_{10}$ as
\begin{equation}
\label{a2-16}
\begin{aligned}
K_{10}=&-\int_{\mathbb{R}^3}\nabla^{k+2}[u\cdot\nabla\chi]\cdot\nabla^{k+2}\chi dx
\\
\lesssim&\|\nabla^{k+2}\chi\|_{L^6}\|\nabla^{k+2}(u\cdot\nabla\chi)\|_{L^{\frac65}}
\\
\lesssim&\|\nabla^{k+2}\nabla\chi\|_{L^2}(\|\nabla^{k+2}u\|_{L^2}\|\nabla\chi\|_{L^3}+\|u\|_{L^3}\|\nabla^{k+2}\nabla\chi\|_{L^2})
\\
\lesssim&(\|\Lambda^{\frac32}\chi\|_{L^2}+\|\Lambda^{\frac12}u\|_{L^2})(\|\nabla^{k+2}\nabla\chi\|_{L^2}^2+\|\nabla^{k+2}u\|_{L^2}^2)
\\
\lesssim&\delta(\|\nabla^{k+2}\nabla\chi\|_{L^2}^2+\|\nabla^{k+2}u\|_{L^2}^2) .
\end{aligned}
\end{equation}
Using H\"{o}lder's inequality, Kato-Ponce inequality, Gagliado-Nirenberg inequality together with Sobolev embedding theorem again, we have
\begin{equation}
\begin{aligned}\label{a2-20}
K_{11}=&-\int_{\mathbb{R}^3}\nabla^{k+2}[\varphi(\varrho)\Delta\chi]\cdot\nabla^{k+2}\chi dx
\\
\lesssim&\|\nabla^{k+3}\chi\|_{L^2}\|\nabla^{k+1}[\varphi(\varrho)\Delta\chi]\|_{L^2}
\\
\lesssim&\|\nabla^{k+3}\chi\|_{L^2}(\|\varphi(\varrho)\|_{L^{\infty}}\|\nabla^{k+1}\Delta\chi\|_{L^2}+\|\Delta\chi\|_{L^{\infty}}\|\nabla^{k+1}
\varphi(\varrho)\|_{L^2})
\\
\lesssim&\|\nabla^{k+3}\chi\|_{L^2}(\|\nabla\varrho\|_{L^{2}}^{\frac12}\|\nabla^2\varrho\|_{L^{2}}
^{\frac12}\|\nabla^{k+3}\chi\|_{L^2}+\|\nabla^3\chi\|_{L^2}^{\frac12}\|\nabla^3\nabla\chi\|_{L^2}^{\frac12}\|\nabla^{k+1}\varrho\|_{L^2})
\\
\lesssim&\delta(\|\nabla^{k+3}\chi\|_{L^2}^2+\|\nabla^{k+1}\varrho\|_{L^2}^2),
\end{aligned}
\end{equation}
and
\begin{equation}
\begin{aligned}
\label{a2-21}
K_{12}=&-\int_{\mathbb{R}^3}\nabla^{k+2}[\phi(\varrho)(\chi^3-\chi)]\cdot\nabla^{k+2}\chi dx
\\
\lesssim&\|\nabla^{k+3}\chi\|_{L^2}(\|\nabla^{k+1}[\phi(\varrho) \chi^3 ]\|_{L^2}+\|\nabla^{k+1}[\zeta(\varrho)\zeta(\varrho) \chi ]\|_{L^2})
\\
\lesssim&\|\nabla^{k+3}\chi\|_{L^2}(\|\phi(\varrho)\|_{L^{\infty}}\|\chi\|^2_{L^6}\|\nabla^{k+1}
 \chi \|_{L^6}+\| \chi\|_{L^{\infty}}\|\nabla^{k+1}\phi(\varrho)\|_{L^2}
 \\&+\|\zeta(\varrho)\|_{L^6}^2\|\nabla^{k+1}\chi\|_{L^6}+\|\zeta(\varrho)\|_{L^\infty}\|\chi\|_{L^\infty}\|\nabla^{k+1}\zeta(\varrho)\|_{L^2})
 \\
 \lesssim&\|\nabla^{k+3}\chi\|_{L^2}(\|\nabla\varrho \|_{L^{2}}^{\frac12}\|\Delta\varrho\|_{L^2}^{\frac12}\|\nabla\chi\|^2_{L^2}\|\nabla^{k+1}
 \chi \|_{L^6}+\| \nabla\chi\|_{L^{2}}^{\frac12}\| \Delta\chi\|_{L^{2}}^{\frac12}\|\nabla^{k+1} \varrho \|_{L^2}
 \\&+\|\nabla\varrho \|_{L^2}^2\|\nabla^{k+2}\chi\|_{L^2}+\| \nabla\chi\|_{L^{2}}^{\frac12}\| \Delta\chi\|_{L^{2}}^{\frac12}\|\nabla^{k+1} \varrho \|_{L^2})
 \\
 \lesssim&\left(\delta^3+\delta\right)(\|\nabla^{k+3}\chi\|_{L^2}^2+\|\nabla^{k+1}\varrho \|_{L^2}^2) 
.
\end{aligned}
\end{equation}
Summing up the estimates for $K_1$ -$K_{12}$,   we deduce (\ref{a2-2}), this yields the desired result.
\end{proof}

The following lemma provides the dissipation estimate for $\varrho$.
\begin{lemma}
\label{lem2.2}
If $\sqrt{\mathcal{E}_0^3(t)}<\delta$, then for $k=0,1,\cdots,N-1$, we have
\begin{equation}
\label{2-25}\begin{aligned}
\frac d{dt}\int_{\mathbb{R}^3}\nabla^ku\cdot\nabla^{k+1}\varrho dx+C \|\nabla^{k+1}\varrho\|_{L^2}^2
\lesssim
\|\nabla^{k+1}u\|_{L^2}^2+\|\nabla^{k+2}u\|_{L^2}^2%+\|\nabla^{k+2} \chi\|_{L^2}^2
+\|\nabla^{k+3} \chi\|_{L^2}^2.\end{aligned}
\end{equation}
\end{lemma}
\begin{proof}
Applying $\nabla^k$ to (\ref{1-2})$_2$, multiplying $\nabla\nabla^k\varrho$, integrating over $\mathbb{R}^3$ by parts, it yields that
\begin{equation}
\label{2-26}
\begin{aligned}
\int_{\mathbb{R}^3}|\nabla^{k+1}\varrho|^2dx\leq&-\int_{\mathbb{R}^3}\nabla^ku_t\cdot\nabla\nabla^k\varrho dx+C\|\nabla^{k+2}u\|_{L^2}\|\nabla^{k+1}\varrho\|_{L^2}
\\
&+\left\|\nabla^k\left[u\cdot\nabla u+h(\varrho)(\mu\Delta u+(\mu+\lambda)\nabla\hbox{div}u)+g(\varrho)\nabla\varrho\right.\right.
\\
&\left.\left.+\psi(\varrho)\hbox{div} \left(\nabla\chi\otimes\nabla\chi-\frac{|\nabla\chi|^2}2 \mathbb{I}_3\right)\right]\right\|_{L^2}\|\nabla^{k+1}\varrho\|_{L^2}.
\end{aligned}
\end{equation}
For the first term of the right hand side of (\ref{2-26}), using (\ref{1-2})$_1$, integrating by parts for both the $t$ and $x$ variables, one obtain
\begin{equation}
\begin{aligned}
\label{2-27}&
-\int_{\mathbb{R}^3}\nabla^ku_t\cdot\nabla\nabla^k\varrho dx
\\
=&-\frac d{dt}\int_{\mathbb{R}^3}\nabla^ku \cdot\nabla\nabla^k\varrho dx-\int_{\mathbb{R}^3}\nabla^k\hbox{div}u\cdot\nabla^k\varrho_tdx
\\
=&-\frac d{dt}\int_{\mathbb{R}^3}\nabla^ku \cdot\nabla\nabla^k\varrho dx+\|\nabla^k\hbox{div}u\|_{L^2}^2+\int_{\mathbb{R}^3}\nabla^k\hbox{div}u\cdot\nabla^k\hbox{div}(\varrho u)dx.
\end{aligned}
\end{equation}
Employing   H\"{o}lder's inequality, Kato-Ponce inequality and Sobolev embedding theorem, we obtain
\begin{equation}
\begin{aligned}\label{2-28}&
\int_{\mathbb{R}^3}\nabla^k\hbox{div}u\cdot\nabla^k\hbox{div}(\varrho u)dx
\\
\lesssim&\|\nabla^{k+1}u\|_{L^2}(\|\nabla^{k+1}\varrho\|_{L^2}\|u\|_{L^\infty}+\|\varrho\|_{L^\infty}\|\nabla^{k+1}u\|_{L^2})
\\
\lesssim&\|\nabla^{k+1}u\|_{L^2}(\|\nabla^{k+1}\varrho\|_{L^2}\|\nabla u\|_{L^2}^{\frac12}\|\Delta u\|_{L^2}^{\frac12}+\|\nabla\varrho\|_{L^2}^{\frac12}\|\Delta\varrho\|_{L^2}^{\frac12}\|\nabla^{k+1}u\|_{L^2})
\\
\lesssim&(\|\nabla u\|_{L^2}+\|\Delta u\|_{L^2}+\|\nabla\varrho\|_{L^2}+\|\Delta\varrho\|_{L^2})(\|\nabla^{k+1}u\|_{L^2}^2+
\|\nabla^{k+1}\varrho\|_{L^2}^2)
\\
\lesssim&\delta(\|\nabla^{k+1}u\|_{L^2}^2+
\|\nabla^{k+1}\varrho\|_{L^2}^2).\end{aligned}
\end{equation}
It then follows from (\ref{2-27}) and (\ref{2-28}) that
\begin{equation}
\label{2-32}
\begin{aligned}&
-\int_{\mathbb{R}^3}\nabla^ku_t\cdot\nabla\nabla^k\varrho dx
\\
\leq&-\frac d{dt}\int_{\mathbb{R}^3}\nabla^ku \cdot\nabla\nabla^k\varrho dx+C(\|\nabla^{k+1}u\|_{L^2}^2+\|\nabla^{k+2}u\|_{L^2}^2)+C\delta\|\nabla^{k+1}\varrho\|_{L^2}^2.
\end{aligned}
\end{equation}
Next, we also need to estimate the last term of the right hand side of (\ref{2-26}).
Note that
\begin{equation}
\begin{aligned}\label{2-33}
 \|\nabla^k(u\cdot\nabla u)\|_{L^2}
 \lesssim&\|u\|_{L^3}\|\nabla^{k+1}u\|_{L^6}+\|\nabla u\|_{L^3}\|\nabla^ku\|_{L^6}
 \\
 \lesssim&(\|\Lambda^{\frac12}u\|_{L^2}+\|\Lambda^{\frac32}u\|_{L^2})(\|\nabla^{k+1}u\|_{L^2}+\|\nabla^{k+2}u\|_{L^2})
 \\
 \lesssim&\delta(\|\nabla^{k+1}u\|_{L^2}+\|\nabla^{k+2}u\|_{L^2}) .
 \end{aligned}\end{equation}
We also have
\begin{equation}
\begin{aligned}\label{2-34}&
\|\nabla^k[h(\varrho)(\mu\Delta u+(\mu+\lambda)\nabla\hbox{div}u)]\|_{L^2}\\\approx&\|\nabla^k(h(\varrho)\nabla^2 u\|_{L^2}
\\
\lesssim&\|h(\varrho)\|_{L^{\infty}}\|\nabla^{k+2}u\|_{L^2}+\|\nabla^2u\|_{L^3}\|\nabla^kh(\varrho)\|_{L^6}
\\
\lesssim&\|\varrho\|_{L^{\infty}}\|\nabla^{k+2}u\|_{L^2}+\|\nabla^2u\|_{L^3}\|\nabla^{k+1}\varrho \|_{L^2}
\\
\lesssim&\delta(\|\nabla^{k+2}u\|_{L^2}+\|\nabla^{k+1}\varrho \|_{L^2}),
\end{aligned}
\end{equation}
and
\begin{equation}
\begin{aligned}\label{2-35}
\|\nabla^k(g(\varrho)\nabla\varrho)\|_{L^2}\lesssim &\|g(\varrho)\|_{L^{\infty}}\|\nabla^{k+1}\varrho\|_{L^2}+\|\nabla\varrho\|_{L^3}\|\nabla^kg(\varrho)\|_{L^6}
\\
\lesssim&\|\varrho\|_{L^{\infty}}\|\nabla^{k+1}\varrho\|_{L^2}+\|\nabla\varrho\|_{L^3}\|\nabla^{k+1}\varrho\|_{L^2}
\\
\lesssim&\delta \|\nabla^{k+1}\varrho \|_{L^2}.
\end{aligned}
\end{equation}
Moreover, Kato-Ponce inequality of Lemma \ref{Kato} and Sobolev inequality of Lemma \ref{lem2.1A} imply that
\begin{equation}
\begin{aligned}\label{2-36}
&\left\|\nabla^k\left[\psi(\varrho)\hbox{div} \left(\nabla\chi\otimes\nabla\chi-\frac{|\nabla\chi|^2}2 \mathbb{I}_3\right)\right]\right\|_{L^2}
\\
\lesssim&\|\psi(\varrho)\|_{L^3}\left\|\nabla^k \hbox{div} \left(\nabla\chi\otimes\nabla\chi-\frac{|\nabla\chi|^2}2 \mathbb{I}_3\right) \right\|_{L^6}
+\left\| \hbox{div} \left(\nabla\chi\otimes\nabla\chi-\frac{|\nabla\chi|^2}2 \mathbb{I}_3\right)\right\|_{L^3}\|\nabla^k\psi(\varrho)\|_{L^6}
\\
\approx&\|\psi(\varrho)\|_{L^3} \|\nabla^{k+1}|\nabla\chi|^2  \|_{L^6}
+ \||\nabla\chi||\nabla^2\chi|\|_{L^3}\|\nabla^k\psi(\varrho) \|_{L^6}
\\
\lesssim&\|\varrho\|_{L^3}\|\nabla\chi\|_{L^{\infty}}\|\nabla^{k+2}\nabla\chi\|_{L^2}+\|\nabla\chi\|_{L^6}
\|\nabla^2\chi\|_{L^6}\|\nabla^{k+1}\varrho\|_{L^2}
\\
\lesssim&\delta^3(\|\nabla^{k+1}\varrho \|_{L^2}+\|\nabla^{k+2}\nabla\chi\|_{L^2}).
\end{aligned}
\end{equation}
Combining (\ref{2-33})-(\ref{2-36}) together, we easily obtain
\begin{equation}
\begin{aligned}\label{2-37}&
\left\|\nabla^k\left[u\cdot\nabla u+h(\varrho)(\mu\Delta u+(\mu+\lambda)\nabla\hbox{div}u)+g(\varrho)\nabla\varrho\right.\right.
\\
&\left.\left.+\psi(\varrho)\hbox{div} \left(\nabla\chi\otimes\nabla\chi-\frac{|\nabla\chi|^2}2 \mathbb{I}_3\right)\right]\right\|_{L^2}\|\nabla^{k+1}\varrho\|_{L^2}
\\
\lesssim&(\delta^3+\delta)(\|\nabla^{k+1}\varrho \|_{L^2}+\|\nabla^{k+1}u\|_{L^2}+\|\nabla^{k+2}u\|_{L^2}+\|\nabla^{k+2}\nabla\chi\|_{L^2}).
\end{aligned}
\end{equation}
Plugging the estimates (\ref{2-32}) and (\ref{2-37}) into (\ref{2-26}), by using Cauchy's inequality and the smallness of $\delta$, we then complete the proof of Lemma \ref{lem2.2}.
\end{proof}

\section{Negative Sobolev estimates}

In this section, we derive the evolution of the negative Sobolev norms of the solution.
\begin{lemma}
\label{lem3.1}If $\sqrt{\mathcal{E}_0^3(t)}\leq\delta$. Then for $s\in[0,\frac12]$, we have
\begin{equation}\begin{aligned}
\label{3-1}&
\frac d{dt}(\|\Lambda^{-s}\varrho\|_{L^2}^2+\|\Lambda^{-s}u\|_{L^2}^2+\|\Lambda^{-s}\chi\|_{L^2}^2+\|\Lambda^{-s}\nabla\chi\|_{L^2}^2)
\\&+C(\|\nabla\Lambda^{-s}u\|_{L^2}^2+\|\nabla\Lambda^{-s}\chi\|_{L^2}^2+\|\nabla\Lambda^{-s}\chi\|_{L^2}^2)
\\
\lesssim&(\|\varrho\|_{H^2}^2+\|\nabla u\|_{H^1}^2+\|\nabla\chi\|_{H^3}^2)(\|\Lambda^{-s}\varrho\|_{L^2} +\|\Lambda^{-s}u\|_{L^2} +\|\Lambda^{-s}\chi\|_{L^2} +\|\Lambda^{-s}\nabla\chi\|_{L^2} ).
\end{aligned}
\end{equation}
Moreover, for $s\in(\frac12,\frac32)$, we have
\begin{equation}\begin{aligned}
\label{3-1z}&
\frac d{dt}(\|\Lambda^{-s}\varrho\|_{L^2}^2+\|\Lambda^{-s}u\|_{L^2}^2+\|\Lambda^{-s}\chi\|_{L^2}^2+\|\Lambda^{-s}\nabla\chi\|_{L^2}^2)
\\&+C(\|\nabla\Lambda^{-s}u\|_{L^2}^2+\|\nabla\Lambda^{-s}\chi\|_{L^2}^2+\|\nabla\Lambda^{-s}\chi\|_{L^2}^2)
\\
\lesssim&\left[\|(\varrho,u,\chi,\nabla\chi)\|_{L^2}^{s-\frac12}(\|\varrho\|_{H^2} +\|\nabla u\|_{H^1} +\|\nabla\chi\|_{H^1}  +\|\nabla^2\chi\|_{H^1} )^{\frac52-s}+(\|\nabla\varrho\|_{H^1}^2+\|\Delta\chi\|_{L^2}^2)\right]
\\
&\times(\|\Lambda^{-s}\varrho\|_{L^2} +\|\Lambda^{-s}u\|_{L^2} +\|\Lambda^{-s}\chi\|_{L^2} +\|\Lambda^{-s}\nabla\chi\|_{L^2} ).
\end{aligned}
\end{equation}
\end{lemma}
\begin{proof}
Applying $\Lambda^{-s}$ to (\ref{1-2})$_1$, (\ref{1-2}$_2$ and (\ref{1-2})$_3$, $\Lambda^{-s}\nabla$ to (\ref{1-2})$_3$, multiplying the resulting identities by $\Lambda^{-s}\varrho$, $\Lambda^{-s}u$, $\Lambda^{-s}\chi$ and $\Lambda^{-s}\nabla\chi$, respectively, summing up and then integrating by parts, we deduce that
\begin{equation}
\label{3-2}
\begin{aligned}&
\frac12\frac d{dt}\int_{\mathbb{R}^3}(|\Lambda^{-s}\varrho|^2+|\Lambda^{-s}u|^2+|\Lambda^{-s}\chi|^2+|\Lambda^{-s}\nabla\chi|^2)dx
\\
&+\int_{\mathbb{R}^3}(|\Lambda^{-s}\nabla u|^2+|\Lambda^{-s}\nabla \chi|^2+|\Lambda^{-s}\nabla ^2\chi|^2)dx
\\
=&\int_{\mathbb{R}^3}\left[\Lambda^{-s} (-\varrho\hbox{div}u-u\cdot\nabla\varrho)\cdot\Lambda^{-s} \varrho \right.
\\&
\left.-\Lambda^{-s} \left[u\cdot\nabla u+h(\varrho)(\mu\Delta u+(\mu+\lambda)\nabla\hbox{div}u)+g(\varrho)\nabla\varrho+\phi(\varrho)\hbox{div}\left(\nabla\chi\otimes\nabla\chi-\frac{|\nabla\chi|^2}2 \mathbb{I}_3\right)\right]\cdot\Lambda^{-s}u \right.
\\
&\left.
+\Lambda^{-s} \left(-u\cdot\nabla\chi-\varphi(\varrho)\Delta\chi-\phi(\varrho)(\chi^3-\chi) \right)\cdot\Lambda^{-s} \chi  \right.
\\
&\left.
+\Lambda^{-s} \nabla\left(-u\cdot\nabla\chi-\varphi(\varrho)\Delta\chi-\phi(\varrho)(\chi^3-\chi)\right)\cdot\Lambda^{-s} \nabla\chi  \right]dx
\\
=&\sum_{i=1}^{12}J_i.
\end{aligned}
\end{equation}
The main tool to estimate the nonlinear terms in the right hand side of (\ref{3-2}) is the estimate in Lemma \ref{lem2.3A}. This forces us to require that $s\in(0,\frac32)$. If $s\in(0,\frac12]$, we easily obtain $\frac12+\frac s3<1$ and $\frac 3s\geq6$. Then, applying Lemma \ref{lem2.3A}, Lemma \ref{lem2.1A} together with H\"{o}lder's and Young's inequalities, it yields that
\begin{equation}
\label{3-4c}\begin{aligned}
J_1=&-\int_{\mathbb{R}^3}\Lambda^{-s}(\varrho\nabla\cdot u)\Lambda^{-s}\varrho dx\lesssim\|\Lambda^{-s}\varrho\|_{L^2}\|\Lambda^{-s}(\varrho\nabla\cdot u)\|_{L^2}
\\
\lesssim&\|\Lambda^{-s}\varrho\|_{L^2}\|\varrho\nabla\cdot u\|_{L^{\frac1{\frac12+\frac s3}}}\lesssim \|\Lambda^{-s}\varrho\|_{L^2}\|\varrho\|_{L^{\frac 3s}}\|\nabla u\|_{L^2}
\\
\lesssim&\|\Lambda^{-s}\varrho\|_{L^2}\|\nabla\varrho\|_{L^2}^{\frac12+s}\|\nabla^2\varrho\|_{L^2}^{\frac12-s}\|\nabla u\|_{L^2}
\\
\lesssim&(\|\nabla \varrho\|_{H^1}^2+\|\nabla u\|_{L^2}^2)\|\Lambda^{-s}\varrho\|_{L^2}.
\end{aligned}\end{equation}
Similarly, by using Lemma \ref{lem2.3A}, Lemma \ref{lem2.1A} together with H\"{o}lder's and Young's inequalities, the term $J_2$-$J_{12}$ can be bound by
\begin{equation}
\label{3-4r1}\begin{aligned}
J_2=&-\int_{\mathbb{R}^3}\Lambda^{-s}(u\cdot\nabla\varrho )\Lambda^{-s}\varrho dx\lesssim\|\Lambda^{-s}\varrho\|_{L^2}\|\Lambda^{-s}(u\cdot\nabla\varrho )\|_{L^2}
\\
\lesssim&\|\Lambda^{-s}\varrho\|_{L^2}\|u\cdot\nabla\varrho \|_{L^{\frac1{\frac12+\frac s3}}}\lesssim \|\Lambda^{-s}\varrho\|_{L^2}\|u\|_{L^{\frac 3s}}\|\nabla \varrho\|_{L^2}
\\
\lesssim&\|\Lambda^{-s}\varrho\|_{L^2}\|\nabla u\|_{L^2}^{\frac12+s}\|\nabla^2u\|_{L^2}^{\frac12-s}\|\nabla \varrho\|_{L^2}
\\
\lesssim&(\|\nabla u\|_{H^1}^2+\|\nabla \varrho\|_{L^2}^2)\|\Lambda^{-s}\varrho\|_{L^2},
\end{aligned}\end{equation}
\begin{equation}
\label{3-4r2}\begin{aligned}
J_3=&-\int_{\mathbb{R}^3}\Lambda^{-s}(u\cdot\nabla u )\Lambda^{-s}u dx\lesssim\|\Lambda^{-s}u\|_{L^2}\|\Lambda^{-s}(u\cdot\nabla u )\|_{L^2}
\\
\lesssim&\|\Lambda^{-s}u\|_{L^2}\|u\cdot\nabla u \|_{L^{\frac1{\frac12+\frac s3}}}\lesssim \|\Lambda^{-s} u\|_{L^2}\|u\|_{L^{\frac 3s}}\|\nabla u\|_{L^2}
\\
\lesssim&\|\Lambda^{-s}u\|_{L^2}\|\nabla u\|_{L^2}^{\frac12+s}\|\nabla^2u\|_{L^2}^{\frac12-s}\|\nabla u\|_{L^2}
\\
\lesssim&\|\nabla u\|_{H^1}^2 \|\Lambda^{-s}u\|_{L^2},
\end{aligned}\end{equation}
\begin{equation}
\label{3-4r4}\begin{aligned}
J_4=&-\int_{\mathbb{R}^3}\Lambda^{-s}(h(\varrho)(\nu\Delta u+(\nu+\eta)\nabla\hbox{div}u) )\Lambda^{-s}u dx\\\lesssim&\|\Lambda^{-s}u\|_{L^2}\|\Lambda^{-s}(h(\varrho)(\nu\Delta u+(\nu+\eta)\nabla\hbox{div}u) )\|_{L^2}
\\
\lesssim&\|\Lambda^{-s}u\|_{L^2}\| h(\varrho)(\nu\Delta u+(\nu+\eta)\nabla\hbox{div}u)  \|_{L^{\frac1{\frac12+\frac s3}}}\\\lesssim& \|\Lambda^{-s} u\|_{L^2}\|h(\varrho)\|_{L^{\frac 3s}}\|\Delta u\|_{L^2}
\\
\lesssim&\|\Lambda^{-s}u\|_{L^2}\|\nabla \varrho\|_{L^2}^{\frac12+s}\|\nabla^2\varrho\|_{L^2}^{\frac12-s}\|\nabla^2 u\|_{L^2}
\\
\lesssim&(\|\nabla \varrho\|_{H^1}^2+\|\nabla^2 u\|_{L^2}^2)\|\Lambda^{-s}u\|_{L^2},
\end{aligned}\end{equation}
\begin{equation}
\label{3-4a3}\begin{aligned}
J_5=&-\int_{\mathbb{R}^3}\Lambda^{-s}(g(\varrho)\nabla\varrho )\Lambda^{-s}u dx\lesssim\|\Lambda^{-s}u\|_{L^2}\|\Lambda^{-s}(g(\varrho)\nabla\varrho )\|_{L^2}
\\
\lesssim&\|\Lambda^{-s}u\|_{L^2}\| g(\varrho)\nabla\varrho  \|_{L^{\frac1{\frac12+\frac s3}}}\lesssim \|\Lambda^{-s} u\|_{L^2}\|g(\varrho)\|_{L^{\frac 3s}}\|\nabla\varrho\|_{L^2}
\\
\lesssim&\|\Lambda^{-s}u\|_{L^2}\| \nabla \varrho \|_{L^2}^{\frac12+s}\| \nabla^2\varrho \|_{L^2}^{\frac12-s}\|\nabla\varrho\|_{L^2}
\\
\lesssim&  \|\nabla \varrho\|_{H^1}^2  \|\Lambda^{-s}u\|_{L^2},
\end{aligned}\end{equation}
\begin{equation}
\begin{aligned}\label{3-4}
J_6=&\int_{\mathbb{R}^3}\Lambda^{-s}\left[\phi(\varrho)\hbox{div}\left(\nabla\chi\otimes\nabla\chi-\frac{|\nabla\chi|^2}2 \mathbb{I}_3\right)\right]\cdot\Lambda^{-s}udx
\\
\lesssim&\|\Lambda^{-s}u\|_{L^2}\|\Lambda^{-s}[\phi(\varrho)|\nabla\chi||\nabla^2\chi|]\|_{L^2}
\lesssim \|\Lambda^{-s}u\|_{L^2}\|\phi(\varrho)|\nabla\chi||\nabla^2\chi|\|_{L^{\frac1{\frac12+\frac s3}}}
\\
\lesssim&\|\Lambda^{-s}u\|_{L^2}\||\nabla\chi||\nabla^2\chi|\|_{L^{\frac1{\frac12+\frac s3}}}
\lesssim \|\nabla\chi\|_{L^{\frac3s}}\|\nabla^2\chi\|_{L^2}\|\Lambda^{-s}u\|_{L^2}
\\
\lesssim&\|\nabla^2\chi\|_{L^2}^{\frac12-s}\|\nabla^3\chi\|_{L^2}^{\frac12+s}\|\nabla^2\chi\|_{L^2}\|\Lambda^{-s}u\|_{L^2}
\\
\lesssim&(\|\nabla^2\chi\|_{L^2}^2+\|\nabla ^3\chi\|_{L^2}^2)\|\Lambda^{-s}u\|_{L^2},
\end{aligned}
\end{equation}
\begin{equation}
\begin{aligned}\label{3-5}
J_7=&\int_{\mathbb{R}^3}\Lambda^{-s}(u\cdot\nabla\chi)\cdot\Lambda^{-s}\chi dx
\\
\lesssim&\|\Lambda^{-s}\chi \|_{L^2}\|\Lambda^{-s}(u\cdot\nabla\chi)\|_{L^2}
\lesssim \|\Lambda^{-s}\chi \|_{L^2}\|u\cdot\nabla\chi\|_{L^{\frac1{\frac12+\frac s3}}}
\\
\lesssim&\|\Lambda^{-s}\chi \|_{L^2}\|u\|_{L^{\frac 3s}}\|\nabla\chi\|_{L^{2}}\lesssim\|\nabla u\|_{L^2}^{\frac12-s}\|\nabla^2u\|_{L^2}^{\frac12+s}\|\nabla \chi\|_{L^2}\|\Lambda^{-s}\chi\|_{L^2}
\\
\lesssim&(\|\nabla u\|_{L^2}^2+\|\nabla^2u\|_{L^2}^2+\|\nabla \chi\|_{L^2}^2)\|\Lambda^{-s}\chi\|_{L^2},
\end{aligned}
\end{equation}
\begin{equation}
\begin{aligned}\label{3-6}
J_8=&\int_{\mathbb{R}^3}\Lambda^{-s}(\varphi(\varrho)\Delta\chi)\cdot\Lambda^{-s}\chi dx
\\
\lesssim&\|\Lambda^{-s}\chi\|_{L^2}\|\Lambda^{-s}(\varphi(\varrho)\Delta\chi)\|_{L^2}
\lesssim \|\Lambda^{-s}\chi\|_{L^2}\|\varphi(\varrho)\Delta\chi\|_{L^{\frac1{\frac12+\frac s3}}}
\\
\lesssim&  \|\Lambda^{-s}\chi\|_{L^2}\|\varphi(\varrho)\|_{L^{\frac 3s}}\|\Delta\chi\|_{L^{2}}  \lesssim\|\nabla \varrho\|_{L^2}^{\frac12-s}\|\nabla^2\varrho\|_{L^2}^{\frac12+s}\|\Delta \chi\|_{L^2}\|\Lambda^{-s}\chi\|_{L^2}
\\
\lesssim&(\|\nabla \varrho\|_{L^2}^2+\|\nabla^2\varrho  \|_{L^2}^2+\|\nabla^2 \chi\|_{L^2}^2)\|\Lambda^{-s}\chi\|_{L^2},
\end{aligned}
\end{equation}
\begin{equation}
\begin{aligned}\label{3-7}
J_9=&\int_{\mathbb{R}^3}\Lambda^{-s}(\phi(\varrho)(\chi^3-\chi) )\cdot\Lambda^{-s}\chi dx
\\
\lesssim&\|\Lambda^{-s}\chi\|_{L^2}\|\Lambda^{-s}(\phi(\varrho)(\chi^3-\chi) )\|_{L^2}
\\
\lesssim&\|\Lambda^{-s}\chi\|_{L^2}(\|\Lambda^{-s}(\phi(\varrho) \chi^3  \|_{L^2}+\|\Lambda^{-s}(\zeta(\varrho)\zeta(\varrho) \chi  \|_{L^2})
\\
\lesssim&\|\Lambda^{-s}\chi\|_{L^2}\left(\|\phi(\varrho) \chi^3  \|_{L^{\frac1{\frac12+\frac s3}}}+\|\zeta(\varrho)\zeta(\varrho) \chi \|_{L^{\frac1{\frac12+\frac s3}}}\right) \\
\lesssim&  \|\Lambda^{-s}\chi\|_{L^2}\left(\|\phi(\varrho)\|_{L^{\infty}}\|\chi \|_{L^{\infty}} \|\chi \|_{L^{2}} \|  \chi \|_{L^\frac 3s}+\|\zeta(\varrho)\|_{L^\frac 3s}\|\zeta(\varrho)\|_{L^\infty}\| \chi  \|_{L^2}\right)
\\
\lesssim&\|\Lambda^{-s}\chi\|_{L^2}\left(\|\nabla\chi  \|_{L^{2}}^{\frac12} \|\Delta\chi  \|_{L^{2}}^{\frac12}\|\chi \|_{L^{2}}\|  \nabla\chi\|_{L^2}^{\frac12-s}\|\nabla^2\chi\|_{L^2}^{\frac12+s}\right.\\&\left.+\|  \nabla\varrho\|_{L^2}^{\frac12-s}\|\nabla^2\varrho\|_{L^2}^{\frac12+s}\|\nabla\varrho  \|_{L^{2}}^{\frac12} \|\Delta\varrho \|_{L^{2}}^{\frac12}\|\chi \|_{L^{2}}\right)
%\\
%\lesssim& \|\nabla \varrho\|_{L^2}^{\frac12-s}\|\nabla^2\varrho\|_{L^2}^{\frac12+s}\|  \chi^2-1\|_{L^{\infty}}\|\chi\|_{L^2}\|\Lambda^{-s}\chi\|_{L^2}
\\
\lesssim&(\delta +1)(\| \nabla \varrho\|_{L^2}^2+\| \Delta \varrho\|_{L^2}^2+\|\nabla^2\chi \|_{L^2}^2 +\| \nabla \chi\|_{L^2}^2)\|\Lambda^{-s}\chi\|_{L^2},
\end{aligned}
\end{equation}
\begin{equation}\begin{aligned}
\label{3-8}
J_{10}=&-\int_{\mathbb{R}^3}\Lambda^{-s} \nabla(u\cdot\nabla\chi)\cdot\Lambda^{-s}\nabla\chi dx
\\
=&-\int_{\mathbb{R}^3}\Lambda^{-s}(\nabla u\cdot\nabla\chi+u\cdot\nabla^2\chi)\cdot\Lambda^{-s}\nabla\chi dx
\\
\lesssim&\|\Lambda^{-s}\nabla\chi\|_{L^2}\|\Lambda^{-s}(\nabla u\cdot\nabla\chi+u\cdot\nabla^2\chi)\|_{L^2}
\\
\lesssim&\|\Lambda^{-s}\nabla\chi\|_{L^2}(\|\nabla u\cdot\nabla\chi\|_{L^{\frac1{\frac12+\frac s3}}}+\|u\cdot\nabla^2\chi\|_{L^{\frac1{\frac12+\frac s3}}})
\\
\lesssim&\|\Lambda^{-s}\nabla\chi\|_{L^2}(\|\nabla\chi\|_{L^{\frac 3s}}\|\nabla u\|_{L^2}+\|\nabla^2\chi\|_{L^2}\|u\|_{L^{\frac 3s}})
\\
\lesssim&\|\Lambda^{-s}\nabla\chi\|_{L^2}(\|\nabla^2\chi\|_{L^{2}}^{\frac12+s}\|\nabla^3\chi\|_{L^2}^{\frac12-s}\|\nabla u\|_{L^2}+\|\nabla^2\chi\|_{L^2}\|\nabla u\|_{L^{2}}^{\frac12+s}\|\nabla^2u\|_{L^2}^{\frac12-s})
\\
\lesssim&\|\Lambda^{-s}\nabla\chi\|_{L^2}(\|\nabla^2\chi\|_{L^2}^2+\|\nabla^3\chi\|_{L^2}^2+\|\nabla u\|_{L^2}^2+\|\nabla^2u\|_{L^2}^2),
\end{aligned}\end{equation}
\begin{equation}
\begin{aligned}\label{3-9}
J_{11}=&-\int_{\mathbb{R}^3}\Lambda^{-s}\nabla(\varphi(\varrho)\Delta\chi)\cdot\Lambda^{-s}\nabla\chi dx
\\
\lesssim&\|\Lambda^{-s}(\varphi(\varrho)\nabla^3\chi+\varphi'(\varrho)\nabla\varrho\nabla^2\chi)\|_{L^2}\|\Lambda^{-s}\nabla\chi\|_{L^2}
\\
\lesssim&(\|\varphi(\varrho)\nabla^3\chi\|_{L^{\frac1{\frac12+\frac s3}}}+\|\varphi'(\varrho)\nabla\varrho\nabla^2\chi\|_{L^{\frac1{\frac12+\frac s3}}})\|\Lambda^{-s}\nabla\chi\|_{L^2}
\\
\lesssim&(\|\varphi(\varrho)\nabla^3\chi\|_{L^{\frac1{\frac12+\frac s3}}}+\| \nabla\varrho\nabla^2\chi\|_{L^{\frac1{\frac12+\frac s3}}})\|\Lambda^{-s}\nabla\chi\|_{L^2}
\\
\lesssim&(\|\nabla^3\chi\|_{L^2}\|\varrho\|_{L^{\frac 3s}}+\|\nabla\varrho\|_{L^2}\|\nabla^2\chi\|_{L^{\frac 3s}})\|\Lambda^{-s}\nabla\chi\|_{L^2}
\\
\lesssim&(\|\nabla^3\chi\|_{L^2}\|\nabla\varrho\|_{L^2}^{\frac12+s}\|\nabla^2\varrho\|_{L^2}^{\frac12-
s}+\|\nabla\varrho\|_{L^2}\|\nabla^3\chi\|^{\frac12+s}_{L^{2}}\|\nabla^4\chi\|^{\frac12-s}_{L^{2}})\|\Lambda^{-s}\nabla\chi\|_{L^2}
\\
\lesssim&(\|\nabla^3\chi\|_{L^2}^2+\|\nabla\varrho\|_{L^2}^2+\|\nabla^2\varrho\|_{L^2}^2+\|\nabla^4\chi\|_{L^2}^2)\|\Lambda^{-s}\nabla\chi\|_{L^2},
\end{aligned}\end{equation}
and
\begin{equation}
\begin{aligned}\label{3-10}
J_{12}=&\int_{\mathbb{R}^3}\Lambda^{-s}\nabla[\phi(\varrho)(\chi^3-\chi)]\cdot\Lambda^{-s}\nabla\chi dx
\\
=&\int_{\mathbb{R}^3}\Lambda^{-s}[\phi'(\varrho)\nabla\varrho(\chi^3-\chi)+\phi(\varrho)(3\chi^2-1)\nabla
\chi]\cdot\Lambda^{-s}\nabla\chi dx
\\
\lesssim&\|\Lambda^{-s}\nabla\chi\|_{L^2}(\|\Lambda^{-s}(\phi'(\varrho)\nabla\varrho \chi ^3 \|_{L^2}+\|\Lambda^{-s}(\phi'(\varrho)\nabla\varrho \chi \|_{L^2}\\&+3\|\Lambda^{-s}(\phi(\varrho)  \chi^2 \nabla\chi)\|_{L^2}+\|\Lambda^{-s}(\phi(\varrho) \nabla\chi)\|_{L^2})
\\
\lesssim&\|\Lambda^{-s}\nabla\chi\|_{L^2}\left(\|\phi'(\varrho)\nabla\varrho  \chi^3 \|_{L^{\frac1{\frac12+\frac 3s}}}+\|\phi'(\varrho)\nabla\varrho  \chi \|_{L^{\frac1{\frac12+\frac 3s}}}\right.\\&\left.+\|\phi(\varrho) \chi^2 \nabla\chi\|_{L^{\frac1{\frac12+\frac s3}}}+\|\zeta(\varrho)\zeta(\varrho) \nabla\chi\|_{L^{\frac1{\frac12+\frac s3}}}\right)
\\
\lesssim&\|\Lambda^{-s}\nabla\chi\|_{L^2}\left(\|\phi'(\varrho)\|_{L^{\infty}}\|\chi \|_{L^{\infty}}^2\|\nabla\varrho\|_{L^2}\|\chi\|_{L^{\frac3s}}+\|\phi'(\varrho)\|_{L^{\infty}} \|\nabla\varrho\|_{L^2}\|\chi\|_{L^{\frac3s}}
\right.\\&\left.+\| \chi \|^2_{L^{\infty}}\|\varrho\|_{L^{\frac3s}}\|\nabla\chi\|_{L^2}+\|\varrho\|_{L^{\frac3s}}\|\nabla\chi\|_{L^2}\right)
\\
\lesssim&\|\Lambda^{-s}\nabla\chi\|_{L^2}\left( \|\chi \|_{L^{\infty}}^2\|\nabla\varrho\|_{L^2}\|\nabla\chi\|_{L^2}^{\frac12+s}\|\nabla^2\chi\|_{L^2}^{\frac12-s}+\|\nabla\varrho\|_{L^2}\|\nabla\chi\|_{L^2}^{\frac12+s}\|\nabla^2\chi\|_{L^2}^{\frac12-s}
\right.\\
&\left.
+\| \chi \|_{L^{\infty}}^2\|\nabla\varrho\|_{L^2}^{\frac12+s}\|\nabla^2\varrho\|_{L^2}^{\frac12-s}\|\nabla\chi\|_{L^2}+\|\nabla\varrho\|_{L^2}^{\frac12+s}\|\nabla^2\varrho\|_{L^2}^{\frac12-s}\|\nabla\chi\|_{L^2}\right)
\\
\lesssim&(1+\delta^2)\|\Lambda^{-s}\nabla\chi\|_{L^2} (|\nabla\varrho\|_{L^2}^2+\|\nabla\chi\|_{L^2}^2+\|\nabla^2\chi\|_{L^2}^2+\|\nabla^2\varrho\|_{L^2}^2).
\end{aligned}\end{equation}
Plugging the estimates (\ref{3-4c})-(\ref{3-10}) into (\ref{3-2}), we obtain (\ref{3-1}).

Next, if $s\in(\frac12,\frac32)$, we can estimate $J_1$-$J_{12}$ in a different way. Since $s\in(\frac12,\frac32)$, it is easy to see that $\frac12+\frac s3<1$ and $2<\frac 3s<6$. Then, Lemma \ref{lem2.3A} and Lemma \ref{lem2.1A} implies that
\begin{equation}
\label{z-1}
\begin{aligned}J_1=& \lesssim\|\Lambda^{-s}\varrho\|_{L^2}\|\Lambda^{-s}(\varrho\nabla\cdot u)\|_{L^2}
\\
\lesssim&\|\Lambda^{-s}\varrho\|_{L^2}\|\varrho\nabla\cdot u\|_{L^{\frac1{\frac12+\frac s3}}}\\\lesssim& \|\Lambda^{-s}\varrho\|_{L^2}\|\varrho\|_{L^{\frac 3s}}\|\nabla u\|_{L^2}
\\
\lesssim&\|\Lambda^{-s}\varrho\|_{L^2}\| \varrho\|_{L^2}^{s-\frac12 }\|\nabla \varrho\|_{L^2}^{\frac32-s}\|\nabla u\|_{L^2} .
\end{aligned}\end{equation}
Similarly, by using Lemma \ref{lem2.3A}, Lemma \ref{lem2.1A}, the term $J_2$-$J_{12}$ can be bound by
\begin{equation}
\label{3-4a1x}\begin{aligned}
J_2\lesssim&\|\Lambda^{-s}\varrho\|_{L^2}\|\Lambda^{-s}(u\cdot\nabla\varrho )\|_{L^2}
 \\\lesssim&\|\Lambda^{-s}\varrho\|_{L^2}\|u\|_{L^{\frac 3s}}\|\nabla \varrho\|_{L^2}
\\
\lesssim&\|\Lambda^{-s}\varrho\|_{L^2}\|  u\|_{L^2}^{s-\frac12 }\|\nabla u\|_{L^2}^{\frac32-s}\|\nabla \varrho\|_{L^2} ,
\end{aligned}\end{equation}
\begin{equation}
\label{3-4a2c}\begin{aligned}
J_3 \lesssim&\|\Lambda^{-s}u\|_{L^2}\|\Lambda^{-s}(u\cdot\nabla u )\|_{L^2}
\\ \lesssim& \|\Lambda^{-s} u\|_{L^2}\|u\|_{L^{\frac 3s}}\|\nabla u\|_{L^2}
\\
\lesssim&\|\Lambda^{-s}u\|_{L^2}\|  u\|_{L^2}^{s-\frac12}\|\nabla u\|_{L^2}^{\frac32-s}\|\nabla u\|_{L^2} ,
\end{aligned}\end{equation}
\begin{equation}
\label{3-4a4x}\begin{aligned}
J_4  \lesssim&\|\Lambda^{-s}u\|_{L^2}\|\Lambda^{-s}(h(\varrho)(\nu\Delta u+(\nu+\eta)\nabla\hbox{div}u) )\|_{L^2}
 \\\lesssim& \|\Lambda^{-s} u\|_{L^2}\|h(\varrho)\|_{L^{\frac 3s}}\|\Delta u\|_{L^2}
\\
\lesssim&\|\Lambda^{-s}u\|_{L^2}\|  \varrho\|_{L^2}^{s-\frac12 }\|\nabla \varrho\|_{L^2}^{\frac32-s}\|\nabla^2 u\|_{L^2} ,
\end{aligned}\end{equation}
\begin{equation}
\label{3-4a3x}\begin{aligned}
J_5 \lesssim&\|\Lambda^{-s}u\|_{L^2}\|\Lambda^{-s}(g(\varrho)\nabla\varrho )\|_{L^2}
\\\lesssim& \|\Lambda^{-s} u\|_{L^2}\|g(\varrho)\|_{L^{\frac 3s}}\|\nabla\varrho\|_{L^2}
\\
\lesssim&\|\Lambda^{-s}u\|_{L^2}\|  \varrho \|_{L^2}^{s-\frac12 }\| \nabla \varrho \|_{L^2}^{\frac32-s}\|\nabla\varrho\|_{L^2},
\end{aligned}\end{equation}
\begin{equation}
\begin{aligned}\label{3-4z}
J_6
\lesssim&\|\Lambda^{-s}u\|_{L^2}\||\nabla\chi||\nabla^2\chi|\|_{L^{\frac1{\frac12+\frac s3}}}
\lesssim \|\nabla\chi\|_{L^{\frac3s}}\|\nabla^2\chi\|_{L^2}\|\Lambda^{-s}u\|_{L^2}
\\
\lesssim&\|\nabla \chi\|_{L^2}^{s-\frac12 }\|\nabla^2\chi\|_{L^2}^{\frac32-s}\|\nabla^2\chi\|_{L^2}\|\Lambda^{-s}u\|_{L^2}
,
\end{aligned}
\end{equation}
\begin{equation}
\begin{aligned}\label{3-5z}
J_7
\lesssim &\|\Lambda^{-s}\chi \|_{L^2}\|u\cdot\nabla\chi\|_{L^{\frac1{\frac12+\frac s3}}}
\lesssim \|\Lambda^{-s}\chi \|_{L^2}\|u\|_{L^{\frac 3s}}\|\nabla\chi\|_{L^{2}}
\\
\lesssim&\|  u\|_{L^2}^{s-\frac12 }\|\nabla u\|_{L^2}^{\frac32-s}\|\nabla \chi\|_{L^2}\|\Lambda^{-s}\chi\|_{L^2}
,\end{aligned}
\end{equation}
\begin{equation}
\begin{aligned}\label{3-6z}
J_8
\lesssim &\|\Lambda^{-s}\chi\|_{L^2}\|\varphi(\varrho)\Delta\chi\|_{L^{\frac1{\frac12+\frac s3}}}
\lesssim   \|\Lambda^{-s}\chi\|_{L^2}\|\varphi(\varrho)\|_{L^{\frac 3s}}\|\Delta\chi\|_{L^{2}}
 \\
 \lesssim&\|  \varrho\|_{L^2}^{s-\frac12 }\|\nabla \varrho\|_{L^2}^{\frac32-s}\|\Delta \chi\|_{L^2}\|\Lambda^{-s}\chi\|_{L^2}
,\end{aligned}
\end{equation}
\begin{equation}
\begin{aligned}\label{3-7z}
J_9
\lesssim& \|\Lambda^{-s}\chi\|_{L^2}\|\phi(\varrho)(\chi^3-\chi) \|_{L^{\frac1{\frac12+\frac s3}}}
\\
\lesssim&  \|\Lambda^{-s}\chi\|_{L^2}\|\phi(\varrho)\|_{L^{2}}(\|\chi \|_{L^{\infty}}^2+1)\|  \chi \|_{L^\frac 3s}
\\
\lesssim &\| \varrho\|_{L^2}(\|\chi \|_{L^{\infty}}^2+1) \|   \chi\|_{L^2}^{s-\frac12 }\|\nabla \chi\|_{L^2}^{\frac32-s}\|\Lambda^{-s}\chi\|_{L^2}
%\\
%\lesssim& \|\nabla \varrho\|_{L^2}^{\frac12-s}\|\nabla^2\varrho\|_{L^2}^{\frac12+s}\|  \chi^2-1\|_{L^{\infty}}\|\chi\|_{L^2}\|\Lambda^{-s}\chi\|_{L^2}
\\
\lesssim&(\delta^2+1)\| \varrho\|_{L^2}  \|   \chi\|_{L^2}^{s-\frac12 }\|\nabla \chi\|_{L^2}^{\frac32-s}\|\Lambda^{-s}\chi\|_{L^2},
\end{aligned}
\end{equation}
\begin{equation}\begin{aligned}
\label{3-8z}
J_{10}
\lesssim& \|\Lambda^{-s}\nabla\chi\|_{L^2}(\|\nabla u\cdot\nabla\chi\|_{L^{\frac1{\frac12+\frac s3}}}+\|u\cdot\nabla^2\chi\|_{L^{\frac1{\frac12+\frac s3}}})
\\
\lesssim&\|\Lambda^{-s}\nabla\chi\|_{L^2}(\|\nabla\chi\|_{L^{\frac 3s}}\|\nabla u\|_{L^2}+\|\nabla^2\chi\|_{L^2}\|u\|_{L^{\frac 3s}})
\\
\lesssim&\|\Lambda^{-s}\nabla\chi\|_{L^2}(\|\nabla \chi\|_{L^{2}}^{s-\frac12}\|\nabla^2\chi\|_{L^2}^{\frac32-s}\|\nabla u\|_{L^2}+\|\nabla^2\chi\|_{L^2}\| u\|_{L^{2}}^{s-\frac12 }\|\nabla u\|_{L^2}^{\frac32-s}),
\end{aligned}\end{equation}
\begin{equation}
\begin{aligned}\label{3-9z}
J_{11}
\lesssim&(\|\varphi(\varrho)\nabla^3\chi\|_{L^{\frac1{\frac12+\frac s3}}}+\| \nabla\varrho\nabla^2\chi\|_{L^{\frac1{\frac12+\frac s3}}})\|\Lambda^{-s}\nabla\chi\|_{L^2}
\\
\lesssim&(\|\nabla^3\chi\|_{L^2}\|\varrho\|_{L^{\frac 3s}}+\|\nabla\varrho\|_{L^{\frac3s}}\|\nabla^2\chi\|_{L^{2}})\|\Lambda^{-s}\nabla\chi\|_{L^2}
\\
\lesssim&(\|\nabla^3\chi\|_{L^2}\| \varrho\|_{L^2}^{s-\frac12 }\|\nabla \varrho\|_{L^2}^{\frac32-
s}+\|\nabla^2\chi\|_{L^{2}}\|\nabla\varrho\|_{L^2}^{s-\frac12} \|\nabla^2\varrho\|^{\frac32-s}_{L^{2}})\|\Lambda^{-s}\nabla\chi\|_{L^2},
\end{aligned}\end{equation}
and
\begin{equation}
\begin{aligned}\label{3-10z}
J_{12}
\lesssim&\|\Lambda^{-s}\nabla\chi\|_{L^2}\left(\|\phi'(\varrho)\nabla\varrho (\chi^3-\chi)\|_{L^{\frac1{\frac12+\frac 3s}}}+\|\phi(\varrho)(3\chi^2-1)\nabla\chi\|_{L^{\frac1{\frac12+\frac s3}}}\right)
\\
\lesssim&\|\Lambda^{-s}\nabla\chi\|_{L^2}\left(\|\phi'(\varrho)\|_{L^{\infty}}(1+\|\chi \|_{L^{\infty}}^2)\|\nabla\varrho\|_{L^2}\|\chi\|_{L^{\frac3s}}
+(1+\|\chi \|_{L^{\infty}}^2)\|\varrho\|_{L^{\frac3s}}\|\nabla\chi\|_{L^2}\right)
\\
\lesssim&\|\Lambda^{-s}\nabla\chi\|_{L^2}\left( (1+\|\chi \|_{L^{\infty}}^2)\|\nabla\varrho\|_{L^2}\| \chi\|_{L^2}^{s-\frac12}\|\nabla\chi\|_{L^2}^{\frac32-s}
+(1+\|\chi \|_{L^{\infty}}^2)\| \varrho\|_{L^2}^{s-\frac12}\|\nabla \varrho\|_{L^2}^{\frac32-s}\|\nabla\chi\|_{L^2}\right)
\\
\lesssim&(1+\delta^2)\|\Lambda^{-s}\nabla\chi\|_{L^2} \left( \|\nabla\varrho\|_{L^2}\| \chi\|_{L^2}^{s-\frac12}\|\nabla\chi\|_{L^2}^{\frac32-s}
+ \| \varrho\|_{L^2}^{s-\frac12}\|\nabla \varrho\|_{L^2}^{\frac32-s}\|\nabla\chi\|_{L^2}\right).
\end{aligned}\end{equation}
Plugging the estimates (\ref{z-1})-(\ref{3-10z}) into (\ref{3-2}), we obtain (\ref{3-1z}). Hence, the proof is complete.

\end{proof}

\section{Proof of Theorem \ref{thm1.1}}

 We first close the energy estimates at each $l$-th level in our weaker sense. Suppose that $N\geq3$ and $0\leq l\leq m-1$ with $1\leq m\leq N$. Summing up the estimates (\ref{2-2}) from $k=l$ to $m-1$, since $\sqrt{\mathcal{E}_0^3(t)}\leq\delta$ is sufficiently small, we arrive at
\begin{equation}
\label{4-1}
\begin{aligned}&
\frac d{dt}\sum_{l\leq k\leq m-1} (\|\nabla^k\varrho\|_{L^2}^2+\|\nabla^ku\|_{L^2}^2+\|\nabla^k\chi\|_{L^2}^2+\|\nabla^k\nabla\chi\|_{ L^2}^2)
\\&+C \sum_{l+1\leq k\leq m}(\|\nabla^{k }u\|_{L^2}^2+\|\nabla^{k }\chi\|_{L^2}^2+\|\nabla^{k+1} \chi\|_{L^2}^2)
\\
\lesssim& (\delta^3+\delta) \sum_{l+1\leq k\leq m}\left( \|\nabla^{k }\varrho\|_{L^2}^2+\|\nabla^{k }u\|_{L^2}^2+\|\nabla^{k }\chi\|_{L^2}^2+\|\nabla^{k+1} \chi\|_{L^2}^2\right).
\end{aligned}
\end{equation}
Moreover, let $k=m-1$ in the estimates (\ref{a2-2}), we obtain
\begin{equation}\label{4-1-1}\begin{aligned}&
\frac d{dt}(\|\nabla^m\varrho\|_{L^2}^2+\|\nabla^mu\|_{L^2}^2+\|\nabla^m\chi\|_{L^2}^2+\|\nabla^{m+1}\chi\|_{L^2}^2)
\\
&+C(\|\nabla^{m+1}u\|_{L^2}^2+\|\nabla^{m+1}\chi\|_{L^2}^2+\|\nabla^{m+2}\chi\|_{L^2}^2)
\\
\lesssim& (\delta^3+\delta)(\|\nabla^m\varrho\|_{L^2}^2+\|\nabla^{m+1}u\|_{L^2}^2+\|\nabla^{m+1}\chi\|_{L^2}^2+\|\nabla^{m+2}\chi\|_{L^2}^2).
\end{aligned}\end{equation}
Combining (\ref{4-1}) and (\ref{4-1-1}) together gives
\begin{equation}
\label{4-1-2}
\begin{aligned}&
\frac d{dt}\sum_{l\leq k\leq m} (\|\nabla^k\varrho\|_{L^2}^2+\|\nabla^ku\|_{L^2}^2+\|\nabla^k\chi\|_{L^2}^2+\|\nabla^k\nabla\chi\|_{ L^2}^2)
\\&+C_1 \sum_{l+1\leq k\leq m+1}(\|\nabla^{k }u\|_{L^2}^2+\|\nabla^{k }\chi\|_{L^2}^2+\|\nabla^{k+1} \chi\|_{L^2}^2)
\\
\lesssim& C_2(\delta^3+\delta) \sum_{l+1\leq k\leq m} \|\nabla^{k }\varrho\|_{L^2}^2 .
\end{aligned}
\end{equation}
Summing up the estimates (\ref{2-25}) from $k=l$ to $m-1$, we obtain
\begin{equation}
\label{4-2}\begin{aligned}&
\frac d{dt}\sum_{l\leq k\leq m-1}\int_{\mathbb{R}^3}\nabla^ku\cdot\nabla^{k+1}\varrho dx+C_3\sum_{l+1\leq k\leq m}\|\nabla^k\varrho\|_{L^2}^2
\\
\lesssim&C_4\left(\sum_{l+1\leq k\leq m+1}
\|\nabla^{k}u\|_{L^2}^2
+\sum_{l+2\leq k\leq m+1}\|\nabla^{k} \nabla\chi\|_{L^2}^2\right).\end{aligned}
\end{equation}
Multiplying (\ref{4-2}) by $C_5\equiv 2C_2\delta(\delta^2+1)/C_3$, adding the resulting inequality with (\ref{4-1}), since $\delta>0$ is sufficiently small, we deduce that there exists a positive constant $C_6$ such that for $0\leq l\leq m-1$,
\begin{equation}
\begin{aligned}\label{4-3}&
\frac d{dt}\left\{\sum_{l\leq k\leq m} (\|\nabla^k\varrho\|_{L^2}^2+\|\nabla^ku\|_{L^2}^2+\|\nabla^k\chi\|_{L^2}^2+\|\nabla^k\nabla\chi\|_{ L^2}^2)
+C_5\sum_{l\leq k\leq m-1}\int_{\mathbb{R}^3}\nabla^ku\cdot\nabla^{k+1}\varrho dx\right\}
\\
&+C_6\left\{\sum_{l+1\leq k\leq m}\|\nabla^{k }\varrho\|_{L^2}^2+\sum_{l+1\leq k\leq m+1}
(\|\nabla^{k}u\|_{L^2}^2+\|\nabla^{k}\chi\|_{L^2}^2+\|\nabla^{k} \nabla\chi\|_{L^2}^2\right\}\leq0.
\end{aligned}\end{equation}
Define $\mathcal{E}_l^m(t)$ to be $\frac1{C_6}$ times the expression under the time derivative in (\ref{4-3}). It is easy to see that since $\delta>0$ is so small, $\mathcal{E}_l^m(t)$ is equivalent to
$$
\|\nabla^l\varrho\|^2_{H^{m-l}}+\|\nabla^lu\|^2_{H^{m-l}}+\|\nabla^l\chi\|^2_{H^{m-l}}+\|\nabla^l\nabla\chi\|^2_{H^{m-l}}.
$$
Then, (\ref{4-3}) can be rewritten as that for $0\leq l\leq m-1$,
\begin{equation}
\label{4-4}
\frac d{dt}\mathcal{E}_l^m(t)+\|\nabla^l\varrho\|_{H^{m-l}}^2+\|\nabla^{l+1}u\|_{H^{m-l}}^2+\|\nabla^{l+1}\chi \|_{H^{m-l}}^2+\|\nabla^{l+1}\nabla\chi \|_{H^{m-l}}^2\leq0.
\end{equation}
Taking $l=0$ and $m=3$ in (\ref{4-4}), integrating directly in time, it yields that
\begin{equation}
\label{4-5}
\|\varrho\|_{H^3}^2+\|u\|_{H^3}^2+\|\chi\|_{H^3}^3+\|\nabla\chi\|_{H^3}^2\lesssim\mathcal{E}_0^3(0)\lesssim\|\varrho_0\|_{H^3}^2
 +\|u_0\|_{H^3}^2+\|\chi_0\|_{H^3}^3+\|\nabla\chi_0\|_{H^3}^2.
 \end{equation}
Then, by a standard continuity argument, this closes the a priori estimates (\ref{2-1}) if at the initial time we assume that $\|\varrho_0\|_{H^3}^2
 +\|u_0\|_{H^3}^2+\|\chi_0\|_{H^3}^3+\|\nabla\chi_0\|_{H^3}^2\leq\delta_0$ is sufficiently small. This in turn allows us to take $l=0$ and $m=N$ in (\ref{4-4}), and then integrate it directly in time to obtain (\ref{1-4}).

Next, we prove the decay rate of solutions for $s\in[0,\frac12]$. Define
$$
\mathcal{E}_{-s}(t)=\|\Lambda^{-s}\varrho\|_{L^2}^2+\|\Lambda^{-s}u\|_{L^2}^2+\|\Lambda^{-s}\chi
\|_{L^2}^2+\|\Lambda^{-s}\nabla\chi\|_{L^2}^2.
$$
Then, integrating in time (\ref{3-1}), by the bound (\ref{1-4}), we derive that
\begin{equation}
\label{4-6}
\begin{aligned}
\mathcal{E}_{-s}(t)\leq&\mathcal{E}_{-s}(0)+C\int_0^t(\|\varrho\|_{H^2}^2+\|\nabla u\|_{H^1}^2+\|\nabla\chi\|_{H^3}^2)\sqrt{\mathcal{E}_{-s}(\tau)}d\tau
\\
\leq&C_0\left(1+\sup_{0\leq\tau\leq t}\sqrt{\mathcal{E}_{-s}(\tau)} \right),
\end{aligned}
\end{equation}
which implies (\ref{1-5}) for $s\in[0,\frac12]$, that is
\begin{equation}
\label{4-7}
\|\Lambda^{-s}\varrho(t)\|_{L^2}^2+\|\Lambda^{-s}u(t)\|_{L^2}^2+\|\Lambda^{-s}\chi(t)\|_{L^2}^2+\|\Lambda^{-s}\nabla\phi(t)\|_{L^2}^2\leq C_0.
\end{equation}
Moreover, if  $l=1,2,\cdots,N-1$, we may use Lemma \ref{lem2.2A} to have
$$
\|\nabla^{l+1}f\|_{L^2}\geq C\|\Lambda^{-s}f\|_{L^2}^{-\frac1{l+s}}\|\nabla^lf\|_{L^2}^{1+\frac1{l+s}}.
$$%and
%$$
%\|\nabla^{l+2}f\|_{L^2}\geq C\|\Lambda^{-s}f\|_{L^2}^{-\frac2{l+s}}\|\nabla^lf\|_{L^2}^{1+\frac2{l+s}}.
%$$
Then, by this facts and (\ref{4-7}), we get
\begin{equation}\label{7-1}
\|\nabla^{l+1}(u,\chi,\nabla\chi)\|_{L^2}^2\geq C_0(\|\nabla^{l}(u,\chi,\nabla\chi)\|_{L^2}^2)^{1+\frac1{k+s}}.
\end{equation}
Thus, for $1=1,2,\cdots,N-1$,
$$
\|\nabla^{l+1}(u,\chi,\nabla\chi)\|_{H^{N-l-1}}^2\geq C_0(\|\nabla^{l}(u,\chi,\nabla\chi)\|_{H^{N-l}}^2)^{1+\frac1{l+s}}.
$$
Thus, we deduce from (\ref{4-4}) with $m=N$ the following inequality
\begin{equation}
\label{7-2}
\frac d{dt}\mathcal{E}_l^N+C_0\left(\mathcal{E}_l^N\right)^{1+\frac1{l+s}}\leq 0,\quad\hbox{for}~l=1,2,\cdots,N-1,
\end{equation}
 which implies
\begin{equation}
\label{7-3}
\mathcal{E}_l^N(t)\leq C_0(1+ t)^{-l-s},\quad\hbox{for}~l=1,2,\cdots,N-1.
\end{equation}
Hence, (\ref{1-5}) holds.

On the other hand, the arguments for $s\in[0,\frac12]$ can not be applied to $s\in(\frac12,\frac32)$. However, observing that $\varrho_0,u_0,\chi_0,\nabla\chi_0\in\dot{H}^{-\frac12}$ hold since $\dot{H}^{-s}\bigcap L^2\subset\dot{H}^{-s'}$ for any $s'\in[0,s]$, we can deduce from what we have proved for (\ref{1-5})-(\ref{1-6}) with $s=\frac12$ that the following estimate holds:
\begin{equation}
\label{7-4}
\|\nabla^l\varrho\|_{H^{N-l}}^2+\|\nabla^lu\|_{H^{N-l}}^2+\|\nabla^l\chi\|_{H^{N-l}}^2+\|\nabla^l\nabla\chi\|_{H^{N-l}}^2\leq C_0(1+t)^{-\frac12-l},\quad \hbox{for}~l=0,1,\cdots,N-1.
\end{equation}
Therefore, we deduce from (\ref{3-1z}) that for $s\in(\frac12,\frac32)$,
\begin{equation}
\begin{aligned}\label{7-5}
\mathcal{E}_{-s}(t)\leq&\mathcal{E}_{-s}(0)+C\int_0^t\|(\varrho,u,\chi,\nabla\chi)\|_{L^2}^{s-\frac12}(\|\varrho\|_{H^2} +\|\nabla u\|_{H^1} +\|\nabla\chi\|_{H^1}  +\|\nabla^2\chi\|_{H^1} )^{\frac52-s} \sqrt{\mathcal{E}_{-s}(\tau)}d\tau
\\
&+C\int_0^t(\|\nabla\varrho\|_{H^1}^2+\|\Delta\chi\|^2_{L^2})\sqrt{\mathcal{E}_{-s}(\tau)}d\tau
\\
\leq&C+C\int_0^t(1+\tau)^{-\frac74-\frac s2}d\tau\sup_{\tau\in[0,t]} \sqrt{\mathcal{E}_{-s}(\tau)}+C\sup_{0\leq\tau\leq t}\sqrt{\mathcal{E}_{-s}(\tau)}
\\
\lesssim&1+\sup_{\tau\in[0,t]} \sqrt{\mathcal{E}_{-s}(\tau)},
\end{aligned}\end{equation}
which implies that (\ref{1-5}) holds for $s\in(\frac12,\frac32)$,
i.e.
\begin{equation}
\label{7-6}\|\Lambda^{-s}\varrho(t)\|_{L^2}^2+\|\Lambda^{-s}u(t)\|_{L^2}^2+\|\Lambda^{-s}\chi(t)\|_{L^2}^2+\|\Lambda^{-s}\nabla\phi(t)\|_{L^2}^2\leq C_0.
\end{equation}
Since we have proved (\ref{7-6}), we may repeat the arguments leading to (\ref{1-6}) for $s\in[0,\frac12]$ to prove that they also hold for $s\in(\frac12,\frac32)$. Therefore, the proof  of Theorem \ref{thm1.1} is complete.

\section*{Acknowledgement}
This paper was supported by the Fundamental Research Funds for the Central Universities (grant   N2005031).  The author would like to thank Prof. Yong Zhou   and Prof. Hao Wu for their encourage and guidance.

}

\end{document}